\documentclass[12pt,reqno]{amsart}

\usepackage{amssymb,hyperref}
\usepackage{amsmath}
\usepackage{etoolbox}

\newtheorem{theorem}{Theorem}[section]
\newtheorem{proposition}[theorem]{Proposition}
\newtheorem{lemma}[theorem]{Lemma}
\newtheorem{corollary}[theorem]{Corollary}

\theoremstyle{definition}
\newtheorem{definition}[theorem]{Definition}
\newtheorem{remark}[theorem]{Remark}

\usepackage[usenames,dvipsnames,svgnames,table]{xcolor}
\usepackage{tikz}
\usepackage{tikz-cd}

\newcommand{\lra}{\longrightarrow}

\renewcommand{\Im}{\operatorname{Im}}

\newcommand{\C}{\mathbb{C}}

\newcommand{\alfa}{\alpha}
\newcommand{\alf}{\alpha}

\newcommand{\om}{\omega}

\newcommand{\M}{\mathsf{M}}
\newcommand{\Mg}{{\mathsf{M}_g}}
\newcommand{\Mgg}{{\overline{\mathsf{M}}_g}}
\newcommand{\A}{\mathsf{A}}

\newcommand{\Z}{{\mathbb{Z}}}

\newcommand{\debar}{\overline{\partial } }

\newcommand{\tcc}[2]{\left[{#1 \above 0pt #2}\right]}

\numberwithin{equation}{section}

\begin{document}

\baselineskip=15pt

\title[Projective structures given by theta]{On the projective structures given by theta}

\author[I. Biswas]{Indranil Biswas}

\address{Department of Mathematics, Shiv Nadar University, NH91, Tehsil
Dadri, Greater Noida, Uttar Pradesh 201314, India}

\email{indranil.biswas@snu.edu.in, indranil29@gmail.com}

\author[L. Vai]{Luca Vai}
	
\address{Dipartimento di Matematica, Universit\`a di Pavia, via
Ferrata 5, I-27100 Pavia, Italy}
	
\email{luca.vai01@universitadipavia.it}



\date{}

\begin{abstract}
Given a compact Riemann surface $C$ and its Jacobian $J_C$, the line in $H^0(J_C,\, 2\Theta)$
orthogonal to the sections of $2\Theta$ vanishing at $0\, \in\, J_C$ produces a natural projective structure
on $C$. We investigate the properties of this projective structure.

\end{abstract}

\maketitle

\tableofcontents

\section{Introduction}\label{sec:introduction}

A projective atlas on a Riemann surface $Y$ is given by a covering of $Y$ by holomorphic coordinate charts 
such that all the transition functions are M\"obius transformations. Two such coverings are called equivalent 
if their union also satisfies this property. A projective structure on $Y$ is an equivalence class of such 
coverings. A projective structure is called canonical if its construction does not require making choices. 

Every Riemann surface admits at least one projective structure. For example, fix a compact Riemann surface $Y$ 
of genus at least two and consider the universal covering $u\,:\,\mathbb{H}\,\lra\, Y$ of it
given by the Poincar\'e-Koebe 
Uniformization Theorem. The covering $u$ is well defined uniquely up to a M\"obius transformation.
Since the automorphisms of $\mathbb H$ are M\"obius transformations, any two coordinate functions on
$Y$ given by a locally defined inverse of $u$ differ by a M\"obius transformation. Consequently,
a covering of $Y$ by such coordinate charts produces a projective structure on $Y$.
Therefore, the Uniformization Theorem 
gives a canonical projective structure on any compact Riemann surface.

As far as we know the above one was the only 
known canonical projective structure until another one was constructed, independently in \cite{EPG}, 
\cite{Lo}, using Hodge theoretic methods. The question was asked whether this projective structure coincides 
with the one given by the Uniformization theorem. This question was resolved in \cite{bcfp}, where it was shown that 
they are actually different. It may be mentioned that the Hodge theoretic projective structure is defined
in an indirect way 
exploiting a canonical bijection between the set of projective structures on $Y$ and a certain affine subspace of 
$H^0(3\Delta, \, K_{Y\times Y}\otimes \mathcal{O}_{Y\times Y}(2\Delta)\vert_{3\Delta})$, where $\Delta\,\subset\, 
Y\times Y$ is the reduced diagonal. The details on this correspondence will be
recalled in Section \ref{section:ps}. 

Subsequently, again using this canonical identification between projective structures
and certain sections over $3\Delta$, another canonical projective structure was 
constructed in \cite{BGV}. This third projective structure is constructed from the line in the space of 
sections of the second power of the theta line bundle on the Jacobian of $Y$ that is orthogonal to the sections 
vanishing at the origin of the Jacobian. We will outline the construction of this projective structure,
together with the construction of the aforementioned Hodge theoretic projective structure,
in Section \ref{section:ps}.

Our aim here is to study the canonical projective structure given by the second power of the theta 
bundle, which we will call the \emph{theta projective structure}. In particular, we will show that this 
projective structure differs from the Hodge theoretic one for every genus. In \cite{BGV} it was shown that 
they are different in the very special case of elliptic curves. To complete the picture, we show this new 
projective structure is also different from the one given by the Uniformization theorem for every genus; this 
is shown in Remark \ref{rem:uniformization}.

To obtain these results, our main strategy is to compute the differential of a family of projective 
structures. This can be done as follows: consider the moduli space $\mathcal{P}_g$ parametrizing classes of 
pairs consisting of a compact genus $g$ Riemann surface and a projective structure on it compatible with the 
complex structure. One can endow $\mathcal P_g$ with the structure of a complex orbifold in such a way that 
the natural forgetful map $\mathcal{P}_g\, \longrightarrow\, \M_g$ is a holomorphic torsor for the holomorphic 
cotangent bundle $\Omega_{\M_g}^{1}$ of the moduli space $\M_g$ of smooth
projective curves of genus $g$. Canonical projective 
structures depending smoothly on the moduli point $[C]$ correspond to $C^\infty$ orbifold sections of the 
torsor $\mathcal{P}_g$ over $\M_g$. Consequently, the differential of a smooth section $\beta:\M_g\,\lra\,
\mathcal{P}_g$ can be interpreted as a $(1,\,1)$ form on $\M_g$:
$$\debar \beta \,\in \, C^\infty(\M_g,\, \Omega_{\M_g}^{1,1}).$$

The connection between (1,\,1) forms and canonical projective structures is actually quite intricate. In fact 
for $g\,\geq\, 5$ it has been proven that the $(1,\,1)$ form $\debar \beta$ determines completely the section 
$\beta$; see \cite{BFPT,BGT,FPT}. To give another example of this connection, denote by $\beta^u$ the section 
of $\mathcal P_g\,\longrightarrow\,\M_g$ corresponding to the projective structure given by the Uniformization 
theorem, and denote by $\beta^\eta$ the section of $\mathcal P_g\, \longrightarrow\, \M_g$ corresponding to 
the Hodge theoretic projective structure. By a result of Zograf and Takhtadzhyan, \cite{ZT}, $\debar \beta^u$ 
is a multiple of the Weil-Petersson form $\omega_{WP}$ on $\M_g$. In \cite{bcfp} it was proved that $\debar 
\beta^\eta$ is a multiple of the pull-back of the Siegel form on the moduli space $A_g$ --- of principally 
polarized abelian varieties of dimension $g$ --- by the Torelli map
$$j\,:\, \M_g \,\lra\, \A_g,$$
that sends a curve to its Jacobian equipped with polarization given by a theta divisor.
Since the Weil-Petersson metric is K\"ahler (in particular, it is nondegenerate) and the
differential of the Torelli map is not injective at the hyperelliptic locus, it follows immediately
that $\beta^\eta \,\neq\, \beta^u$. 

Denote by $\beta^\theta$ the section of $\mathcal P_g\, \longrightarrow\, \M_g$ corresponding to the 
projective structure constructed in \cite{BGV} using the sections of the second power of theta. In this paper 
we establish a connection between $\beta^\theta$ and the second order Theta nullwert map
$$\Theta_g\ :\ \A_g(2,4)\ \lra\ \mathbb{P}^{2^g-1}.$$
Here $\A_g(2,4)$ denotes the moduli space of principally polarized abelian varieties with a theta structure. 
It is actually the quotient of the Siegel space ${\mathbb H}_g$ by a certain finite index subgroup of
$\operatorname{Sp}(2g, \mathbb Z)$ and hence
it has a finite covering map to $\A_g$. 
The map $\Theta_g$ has been deeply studied in relation to the Schottky problem.

Section \ref{sect:descend} is devoted to the study of some aspects of the nullwert map $\Theta_g$, and may
be of independent interest.

We prove that the pull-back $\Theta_g^* \omega_{FS}$ of the Fubini-Study metric on $\mathbb P^{2^g-1}$ to 
$\A_g(2,4)$ has better invariance properties than expected, and in fact it actually descends to a $(1,\,1)$ 
form on $\A_g$. This fact plays a very crucial role in this paper. Indeed, as we show in Section 
\ref{sec:theta-proj-struct}, the restriction of $\Theta_g^\ast \omega_{FS}$ to $\M_g$ coincides
--- up to a multiple --- with $\debar \beta^\theta$. This is interesting because it implies that some
of the structure of $\beta^\theta$ can be extended naturally to $\A_g$, and they can be studied there.

Our first main result is the following (See Theorem \ref {prop:debarisfs} below):

\begin{theorem}
The equality
$$
\overline{\partial}\beta^\theta\ =\ 8\pi j^\ast \Theta_g^\ast \omega_{FS}
$$
holds.
\end{theorem}

As an application of this result, using a simple degeneration
technique we prove that the theta projective structure of \cite{BGV} is
different from the Hodge theoretic projective structure of \cite{EPG}. This is the second main result
of the paper (See Theorem \ref{teo:differ} below):

\begin{theorem}
The projective structure $\beta^\theta$ differs generically from the projective structure $\beta^\eta$, in every
genus.
\end{theorem}

The paper is organized as follows. In Section \ref{section:ps} we recall the basic constructions needed here.
In the first part of Section \ref{section:ps} it is recalled how a projective structure can be seen as an 
algebraic object and we explain families of projective structures. Afterwards, we give the 
definition of the Hodge theoretic projective structure $\beta^\eta$ and the definition of the projective 
structure $\beta^\theta$ introduced in \cite{BGV}. In Section \ref{sect:descend} it is shown that $\Theta_g^\ast 
\omega_{FS}$ actually descends to $\A_g$. This is the content of Proposition \ref {prop:descends}. To get to this 
result we relate $\Theta_g$ and another theta nullwert map, denoted by $\theta^2$. It is shown that, up to a 
unitary projective transformation, $\theta^2$ coincides with the composition of $\Theta_g$ with the second Veronese 
embedding. Using this fact, the proof of Proposition \ref {prop:descends} consists of a careful use of the 
Theta Transformation Formula. In Section \ref{sec:theta-proj-struct}, the differential $\debar \beta^\theta$ 
is computed explicitly. In this Section we also compute $\debar \beta^\eta$ in a new way. In Section 
\ref{differ-debar} it is proved that $\beta^\theta \,\neq\, \beta^\eta$ for every genera. This is a consequence of 
Theorem \ref{teo:differ}, which establishes that $\debar \beta^\theta\,\neq \,\debar \beta^\eta$. The strategy of 
the proof is the following: We recall that in the boundary of $\M_g$ (in the Deligne--Mumford compactification) 
one can find \lq\lq copies\rq\rq of $\M_{1,1}$. Indeed, in the boundary we can find subspaces of the form 
$\{[C\cup_p E]\}_{[E]\,\in\, \M_{1,1}}$, where $C$ is a curve of genus $g-1$ and $C\cup_p E$ is the (singular) 
curve that is obtained identifying $p\,\in\, C$ with the origin of $E$. We show that comparing $\debar 
\beta^\theta$ and $\debar \beta^\eta$ over $\{[C\cup_p E]\}_{[E]\,\in\, \M_{1,1}}$ is
actually equivalent to comparing $\debar 
\beta^\theta$ and $\debar \beta^\eta$ over $\M_{1,1}$. Then, the proof of Theorem \ref{teo:differ} 
is reduced to the case of elliptic curves. This case was already considered in \cite{BGV} where it was proved 
that $\beta^\eta \,\neq\, \beta^\theta$ on $\M_{1,1}$. Here we need a slightly stronger result, that is, we 
need to prove the differentials of the sections are different and not just the sections themselves. This is 
the content of Lemma \ref{lemma:differdebar1}.

\section*{Acknowledgements}

We thank the referees for their careful reading of the manuscript and for their helpful suggestions. The 
second author is deeply thankful to Alessandro Ghigi for many helpful discussions, and for his corrections to 
the manuscript. The second author would also like to thank Riccardo Salvati Manni and Bert van Geemen for very 
useful emails on the topic of theta functions. The second author was partially supported by INdAM-GNSAGA, by 
MIUR PRIN 2022: 20228JRCYB, ``Moduli spaces and special varieties'' and by FAR 2016 (Pavia) ``Variet\`a 
algebriche, calcolo algebrico, grafi orientati e topologici''. The first author is partially supported by a J. 
C. Bose Fellowship (JBR/2023/000003).

\section{Families of projective structures}\label{section:ps}

We begin by introducing notation that will be used throughout here. Firstly, $C$ will denote a compact
Riemann surface of genus $g$; its canonical line bundle will be denoted by $K_C$. For $i\,=\,1,\, 2$,
$$
p_i\, :\, S\, :=\, C\times C\, \longrightarrow\, C
$$
is the natural projection to the $i$-th factor, and $$\Delta \,:=\, \{(x,\, x)\,\,\big\vert\,\, x\, \in\, C\}
\, \subset\, C\times C\,=\, S
$$
is the reduced diagonal. Let
\begin{equation}\label{ei1}
K_S\ \cong\ (p_1^\ast K_C)\otimes (p_2^\ast K_C)
\end{equation}
denote the canonical line bundle of the complex surface $S$.

We recall the definition of a projective structure.

\begin{definition}
A holomorphic atlas $\{(U_i\, z_i)\}_{i\in I}$ of $C$ is called {\it projective} if all its transition
functions satisfy
$$z_i\ =\ \frac{a_{ij}z_j+b_{ij}}{c_{ij}z_j+d_{ij}}$$
for all $i,\, j\, \in\, I$ with $U_i\bigcap U_j\,\not=\, \emptyset$, where
$a_{ij}$, $b_{ij}$, $c_{ij}$ and $d_{ij}$ are locally constant functions on
$U_i\bigcap U_j$. Two projective atlases are considered equivalent if their union is also a projective atlas. A
{\it projective structure} on $C$ is an equivalence class of projective atlases.
\end{definition}

The space $\mathcal{P}(C)$ of projective structures on $C$ is an affine space modelled on 
the space $H^0(C,\, 2K_C)$ of global quadratic differentials. This is explained excellently in \cite{Tyu}.

Denote the line bundle $K_S\otimes {\mathcal O}_S(2\Delta)$ by $K_S(2\Delta)$.
Using Poincar\'e adjunction formula $K_S(2\Delta)\big\vert_{\Delta}$ is trivialized.
There is a canonical section
\begin{equation}\label{a1}
s^C\, \in\, H^0(2\Delta,\, K_S(2\Delta)\big\vert_{2\Delta})
\end{equation}
which is uniquely determined by the following two conditions:
\begin{enumerate}
\item $s^C\big\vert_{\Delta}\,=\,1$, and

\item the involution $(x,\, y)\, \longmapsto\, (y,\, x)$ of $2\Delta \, \subset\, C\times C$ sends
$s^C$ to $-s^C$.
\end{enumerate}
See \cite{BR2}, \cite{bcfp} for details on this. Note that for the isomorphism in
\eqref{ei1}, the above involution $(x,\, y)\, \longmapsto\, (y,\, x)$ of $C\times C$ acts as identity
map on $((p_1^\ast K_C)\otimes (p_2^\ast K_C))\big\vert_\Delta$ and acts as multiplication by
$-1$ on $K_S\big\vert_\Delta$.

Assuming that $g\, \geq\, 2$, consider the short exact sequence of sheaves
$$
0\,\longrightarrow\, K_S\big\vert_\Delta
\,\longrightarrow\, K_S(2\Delta)\big\vert_{3\Delta}\,
\longrightarrow\, K_S(2\Delta)\big\vert_{2\Delta}\,\longrightarrow\, 0
$$
and the associated long exact sequence of cohomologies
$$0\,\longrightarrow\, H^0(C,\, 2K_C)\,=\, H^0(C,\, K_S\big\vert_\Delta)
\,\longrightarrow\, H^0(3\Delta,\, K_S(2\Delta)\big\vert_{3\Delta})
$$
$$
\longrightarrow\, H^0(2\Delta,\, K_S(2\Delta)\big\vert_{2\Delta})\,\longrightarrow\, 0$$
(recall that $g\, \geq\, 2$, so $H^1(C,\, 2K_C)\,=\, 0$). There is a canonical isomorphism between
the affine space $\mathcal{P}(C)$ and the affine subspace of $H^0(K_S(2\Delta)\big\vert_{3\Delta})$
consisting of all sections whose restriction to ${2\Delta}$ coincides with
the section $s^C$ in \eqref{a1}; see \cite[Section 3]{BR1}.

The above isomorphism between the affine space $\mathcal{P}(C)$ and the affine subspace of
$H^0(K_S(2\Delta)\big\vert_{3\Delta})$ consisting of all sections whose restriction to ${2\Delta}$
coincides with the section $s^C$ remains valid for $g\,=\, 1$.

There is also a relative version of the above. Let
\begin{equation}\label{en1}
\pi\, :\, {\mathcal C}\, \longrightarrow\, B
\end{equation}
be a smooth holomorphic family of compact Riemann surfaces of genus $g$. The
relative canonical line bundle on ${\mathcal C}$ for the projection $\pi$ will be denoted by ${\mathcal K}$.
We have the reduced relative diagonal divisor
\begin{equation}\label{en3}
{\bf\Delta}_B\,:=\,\{(x,\,x) \ \mid \ x \,\in\, {\mathcal C}\} \,\subset\, {\mathcal C}\times_B{\mathcal C}.
\end{equation}
Let $\widetilde{p},\, \widetilde{q}\, :\, {\mathcal C}\times_B{\mathcal C}\, \longrightarrow\,\mathcal C$
be the natural projections to the first factor and second factor respectively. Consider the family of complex surfaces
\begin{equation}\label{en4}
\Pi\, :\, {\mathcal C}\times_B{\mathcal C}\, \longrightarrow \,B.
\end{equation}
We have the holomorphic line bundle
\begin{equation}\label{en4b}
{\mathbb L}\, :=\,
({\widetilde p}^*{\mathcal K}\otimes {\widetilde q}^*{\mathcal K})\otimes
{\mathcal O}_{{\mathcal C}\times_B{\mathcal C}}(2{\bf\Delta}_B)
\end{equation}
on ${\mathcal C}\times_B{\mathcal C}$. For any $k\, \geq\, 1$, consider the natural inclusion map
$$
{\mathcal O}_{{\mathcal C}\times_B {\mathcal C}}(-k{\bf\Delta}_B)\otimes {\mathbb L}\, \hookrightarrow\,
{\mathcal O}_{{\mathcal C}\times_B {\mathcal C}}\otimes {\mathbb L}\,=\, {\mathbb L}.
$$
The corresponding quotient ${\mathbb L}/({\mathcal O}_{{\mathcal C}\times_B
{\mathcal C}}(-k{\bf\Delta}_B)\otimes {\mathbb L})$ coincides with $I(k)_*I(k)^*{\mathbb L}$, where
$I(k)$ is the inclusion map of the $(k-1)$-th order infinitesimal neighborhood of ${\bf\Delta}_B$
to ${\mathcal C}\times_B {\mathcal C}$.
Using the map $\Pi$ in \eqref{en4}, construct the direct images 
\begin{equation}\label{en5}
{\mathcal V}\, :=\, \Pi_*({\mathbb L}/({\mathcal O}_{{\mathcal C}\times_B
{\mathcal C}}(-3{\bf\Delta}_B)\otimes {\mathbb L}))\,\longrightarrow\,B
\end{equation}
and
$$
{\mathcal V}_2\, :=\, \Pi_*({\mathbb L}/({\mathcal O}_{{\mathcal C}\times_B
{\mathcal C}}(-2{\bf\Delta}_B)\otimes {\mathbb L}))\,\longrightarrow\, B
$$
which are holomorphic vector bundles over $B$. There is a natural surjective homomorphism
\begin{equation}\label{psi}
\Psi\, :\, {\mathcal V}\, \longrightarrow\, {\mathcal V}_2.
\end{equation}
The vector bundle ${\mathcal V}_2$ has a natural holomorphic section given by the
canonical trivialization of $s^C$ in \eqref{a1} for every curve $C$ in the family $\pi$;
the holomorphic section of ${\mathcal V}_2$ constructed this way will be denoted by $s_0$. Now define
\begin{equation}\label{whcv}
\widehat{\mathcal V}\, :=\, \Psi^{-1}(s_0)\, \subset\, \mathcal V\, ,
\end{equation}
where $\Psi$ is the projection in \eqref{psi}. We note that $\widehat{\mathcal V}$ is a holomorphic
affine bundle over $B$ modelled on the vector bundle
\begin{equation}\label{em}
\pi_*{\mathcal K}^{\otimes 2} \ \longrightarrow\ B,
\end{equation}
where $\pi$ is the projection in \eqref{en1} and
$\mathcal K$ is the relative canonical bundle in \eqref{en4b}.
The $C^\infty$ (respectively, holomorphic) sections of the fiber bundle $\widehat{\mathcal V}
\, \longrightarrow\,B$ in \eqref{whcv} are in a natural bijective correspondence with
the $C^\infty$ (respectively, holomorphic) families of projective structures for the family of
curves $\mathcal C$; see \cite[p.~7, Lemma 2.1]{bcfp}.

Let
$$
\beta\, :\, B\, \longrightarrow\, \widehat{\mathcal V}
$$
be a $C^\infty$ section of the projection $\widehat{\mathcal V}\, \longrightarrow\,
B$ (see \eqref{whcv}). Denote the Dolbeault operator for the holomorphic vector bundle 
$\mathcal V$, defined in \eqref{en5}, by $\overline{\partial}_{\mathcal V}$. Since $\beta$ 
is also a section of $\mathcal V$, and $\Psi$ in \eqref{psi} is holomorphic, it follows that
\begin{equation}\label{k2}
\overline{\partial}_{\mathcal V}(\beta)\, \in\, \Omega^{0,1}(B, \, {\rm kernel}(\Psi))
\,=\, \Omega^{0,1}(B, \, \pi_*{\mathcal K}^{\otimes 2});
\end{equation}
see \cite[p.~7, (2.19)]{bcfp}.

There are natural holomorphic restriction maps 
\begin{equation}\label{eqr}
\begin{tikzcd}
\Pi_\ast\mathbb{L} \arrow[r, "\mathbf{r}"] \arrow[rd, "\mathbf{r}_2"'] & \mathcal{V} \arrow[d, "\Psi"] \\
& \mathcal{V}_2 
\end{tikzcd}
\end{equation}
where $\Pi$, $\mathbb L$ and $\Psi$ are constructed in \eqref{en4}, \eqref{en4b} and \eqref{psi}
respectively. So if $f$ is a section of $\Pi_\ast\mathbb{L}$ such that $\mathbf{r}_2(f)\,=\,s_0$, then
$\mathbf{r}(f)$ is a section of $\widehat{\mathcal V}$, and hence
$\mathbf{r}(f)$ gives a family of projective structures. It may be mentioned that not all
sections of $\widehat{\mathcal V}$ can be obtained in this way \cite[p.~18, Theorem 4.2]{bcfp}.

Nonetheless, the two projective structures $\beta^\theta$ and $\beta^\eta$ that we consider in this
paper are actually obtained in the above way from sections of $\Pi_\ast\mathbb{L}$. We now briefly recall
their constructions now.

Let $C$ be a compact Riemann surface of genus $g\, \geq\, 1$ and $x\, \in\, C$ a point.
We have a natural inclusion map
$$j_x\, :\, H^0(C,\, K_C\otimes \mathcal{O}_C(2\cdot x))\, \hookrightarrow\, H^1(C-\{x\},\, \mathbb{C})
\,\cong\, H^1(C,\,\mathbb{C})\,=\, H^{1,0}(C)\oplus H^{0,1}(C)$$
(note that the condition $g\, \geq\, 1$ is used here).
Since we have $\dim H^0(C,\, K_C\otimes \mathcal{O}_C(2\cdot x))\,=\,g+1$ and $\operatorname{Im}\, (j_x)\,\supset
\,H^{1,0}(C)$, it follows that that $j_x^{-1}(H^{0,1}(C))$ is a line in
$H^0(C,\, K_C\otimes \mathcal{O}_C(2\cdot x))$.
If $(U,\, z)$ is any holomorphic coordinate chart on $C$ containing $x$ such that $z(x)\,=\,0$, then
there is a unique $\eta_z\,\in\, H^0(C,\, K_C\otimes \mathcal{O}_C(2\cdot x))$ satisfying the conditions 
$$\eta_z\,\in\, j_x^{-1}(H^{0,1}(C)), \ \ \ \text{ and }\ \ \ \eta_z\,=\,\left(\frac{1}{z^2}+h(z)\right)dz ,$$
over $U$, with $h$ holomorphic. Notice that the residue of $\eta_z$ at $x$ must be zero by the residue
theorem. This construction yields a meromorphic form $\eta_z$ on $C$ for each coordinate $z$. 
It can be shown that there is a unique $\eta_C\,\in\, H^0(S,\, K_{S}(2\Delta))$ such that
\begin{equation}\label{cbd}
\eta_C\big\vert_{C\times \{x\}}\ \,=\ \,\left(p_1^\ast \eta_z\wedge p_2^\ast dz\right)\big\vert_{C\times\{x\}}
\end{equation}
for all charts $(z,\,U)$ such that $0\,\in\, z(U)$,
where, for each chart $z$, the point $x$ is the one with $z(x)\,=\,0$. We also have $\eta_C\big\vert_{\Delta}\,=\,1$ and
$(\eta_C)|_{2\Delta} \, =\, s^C$ (the canonical section in \eqref{a1}). For more details on $\eta$ see
\cite[Section 2]{EPG}, \cite[Chapter 4, Section 6]{Gu2}, \cite[Section 3]{bcfp}.

We now construct another global section of $K_S(2\Delta)$. Let $J(C)$ be the Jacobian of $C$; denote by $H$ the canonical polarization
on $J(C)$. Consider the map
$$\phi\,:\,C\times C\,\lra\, J(C),\qquad (x,\,y)\,\longmapsto\, \mathcal{O}_C(x-y).$$
One can define intrinsically a line bundle $2\Theta$ on $J(C)$ with $c_1(2\Theta)\,=\,2H$ (see for example \cite[p.~13]{vg}). The vector 
space $H^0(J(C),\,2\Theta)$ has a natural conformal class of Hermitian structures. Also, $2\Theta$ is base--point free. Consider the line 
$\mathbb{S}\,\subset\, H^0(J(C),\,2\Theta)$ orthogonal to all sections vanishing at the origin. It can be shown that $$\phi^\ast 
2\Theta\,\cong\, K_{S}\otimes\mathcal{O}_{S}(2\Delta)$$ (see \cite[Proposition 3.1]{BGV}) and that
\begin{equation}\label{cbd2}
\phi^\ast \mathbb{S}\, = \, \sigma_C\otimes {\mathbb{C}} \, \subset\, H^{0}(S,\, K_{S}(2\Delta))
\end{equation}
is the line generated by a unique element $\sigma_C\, \in\,
H^{0}(S,\, K_{S}(2\Delta))$ with $\sigma_C\big\vert_{\Delta}\,=\,1$ (see \cite[Proposition 3.3]{BGV}).
Again $\sigma_C\big\vert_{2\Delta}\,=\,s^C$ (defined in \eqref{a1}). 

Consider now the sections of $\Pi_\ast \mathbb{L}$ defined by
\begin{equation}\label{eta} 
\eta\,:\,B\,\lra \,\Pi_\ast \mathbb{L},\qquad b\,\longmapsto\, \eta_{C_b}
\end{equation} 
\begin{equation}\label{sigma}
\sigma\,:\,B\,\lra\, \Pi_\ast \mathbb{L},\qquad b\,\longmapsto\, \sigma_{C_b}
\end{equation}
(see \eqref{cbd}, \eqref{cbd2}).
Since $(\eta_C)\big\vert_{2\Delta}\,=\,s^C \,=\, 
(\sigma_C)\big\vert_{2\Delta}$, we have $\mathbf{r}_2(\eta)\,=\, s_0\,=\,\mathbf{r}_2(\sigma)$, and 
therefore
\begin{equation}\label{pr-strs}
\beta^\eta \, : = \, \mathbf{r}(\eta),\qquad \beta^\theta \, := \, \mathbf{r}(\sigma)
\end{equation}
define smooth families of projective structures on $B$. As 
$\sigma_C$ and $\eta_C$ only depend on the isomorphism class of $C$, $\beta^\eta$ and $\beta^\theta$ may 
be interpreted as orbifold sections of the orbifold bundle $\mathcal{V}$ on $\M_g$.

\section{A canonical $(1,1)$ form on $\A_g$}\label{sect:descend}

In this section we study some properties of the Theta nullwert map 
\begin{equation}\label{nm}
\Theta_g\ :\ \mathbb{H}_g\ \longrightarrow\ \mathbb{P}^{2^g-1}
\end{equation}
of \cite{vg}, where
${\mathbb H}_g$ is the Siegel space. In particular we study the 
$(1,1)$ form $\Theta_g^\ast \omega_{FS}$, where $\omega_{FS}$ is the Fubini-Study form on 
$\mathbb{P}^{2^g-1}$. We will show that this form descends to $\A_g$, giving a canonical $(1,1)$ form on 
the moduli space of principally polarized abelian varieties. Recall that if $$Z\ =\ [Z_1:\,\cdots :\,Z_{2^g}]$$ 
are homogeneous coordinates on $\mathbb{P}^{2^g-1}$, then $\omega_{FS}$ may be expressed as
\begin{equation}\label{eq:omegafs}
\omega_{FS}\ = \ \frac{\sqrt{-1}}{2}\partial \debar \log |Z|^2.
\end{equation}

We briefly recall the definition of the map $\Theta_g$. 
Let
\begin{equation}\label{eu}
U\,:=\, \frac{\frac{1}{2} \mathbb Z^g}{\mathbb Z^g}
\end{equation}
For $u\,\in\, U$ and $\tau\,\in\,\mathbb{H}_g$ define the second order theta function with characteristic $u$ as
\begin{equation}
\label{eq:secondorder}
\theta_u(z;\,\tau)\, :=\, \sum_{m\in \mathbb{Z}^g} \exp(2\pi\sqrt{-1} (m+u)^t \tau (m+u) +4\pi\sqrt{-1} (m+u)^t z).
\end{equation}
In the right hand side of \eqref{eq:secondorder}, one takes any representative of $u$ in $\mathbb{R}^g$ modulo $\mathbb{Z}^g$. It is easy to see that the right hand side is
actually independent of the representative chosen.
The second order Theta nullwert map $\Theta_g$ is given by
\begin{equation}\label{tau2}
\Theta_g\ :\ \mathbb{H}_g\ \longrightarrow\ \mathbb{P}^{2^g-1}, \ \ \ 
\tau\, \longmapsto\, (\theta_u(0;\,\tau))_{u\in U},
\end{equation}
where $U$ is defined in \eqref{eu}. We re-labeled the elements of the standard basis of $\C^{2^{^g}}$ as $\{e_\alpha\}_{\alpha
\in U}$. The above notation $U$ will be used throughout. 

We introduce some more notation for this section.
For $\alpha,\,\beta\,\in\, \mathbb{R}^g$, define
$$\alpha\cdot \beta\ \, :=\ \,\sum_{i=1}^g \alpha_i\beta_i.$$
The group $\mathbb{Z}/2\mathbb{Z}$ will often be denoted by $\mathbb{Z}_2$. The class of
$\alpha\,\in\, \mathbb{Z}^g$ in $\mathbb{Z}_2^g$ will be denoted by $[\alpha]$.
For $\alpha,\, \beta\, \in\,
\mathbb{Z}_2^g$, define $$\alpha\cdot \beta\ :=\ \sum_{i=1}^g \alpha_i\beta_i
\ \in\ \mathbb{Z}_2 .$$ Let
\begin{equation}\label{eE}
E\,\, :=\,\, \{(\varepsilon,\,\delta)\,\in\, \mathbb{Z}_2^g\times
\mathbb{Z}_2^g\,\, \big\vert\,\,\, \varepsilon\cdot \delta\,=\,0\}
\end{equation}
be the even characteristics. It is well known that $\# E\,=\, 2^{g-1}(2^g+1)$.

We define the theta functions with real characteristics $c^1, \,c^2\,\in\, \mathbb{R}^g$:
\begin{equation}\label{eq:thetachar}
\theta\tcc{c^1}{c^2}(z;\tau)\,\,:=\,\,\sum_{m\in\mathbb{Z}^g}
\exp ({\pi\sqrt{-1} (m+c^1)^t \tau (m+c^1)+
2\pi\sqrt{-1} (m+c^1)^t(z+c^2)}).
\end{equation}
For $m^1,\, m^2\,\in \,\mathbb{Z}^g$, it is easy to see from \eqref{eq:thetachar} that 
\begin{equation}\label{eq:subtle}
\theta\tcc{\frac{m^1+2k^1}{2}}{\frac{m^2+2k^2}{2}}(z;\tau)\,=\,
(-1)^{m^1\cdot k^2}\theta\tcc{m^1/2}{m^2/2}(z;\tau)\ \ \ \ \ \forall\,\ \, k^1,\,k^2\,\in\, \mathbb{Z}^g.
\end{equation}
Therefore, (recall that $[m]$ is the class of $m$ in $\mathbb{Z}_2^g$) the formula
$$\theta^2_{[m^1],[m^2]}(z;\tau) \, :=\, \theta^2\tcc{m^1/2}{m^2/2}(z;\tau)\ \ \ \forall\,\ \, m^1,\,m^2\,\in\,\mathbb{Z}^g$$
defines, consistently, maps $\theta^2_{\varepsilon,\delta}$ for all pairs $(\varepsilon,\,\delta)\,\in \,\mathbb{Z}_2^g$.
Through these functions define the auxiliary map 
\begin{equation}\label{et2}
\theta_g^2\,\,\,:\,\,\,\mathbb{H}_g\,\longrightarrow\, \mathbb{P}^{2^{g-1}(2^g+1)-1},\ \ \
\tau\,\longmapsto\, \left(\theta^2_{\varepsilon,\delta}(0;\tau)\right)_{\varepsilon,\delta\in E};
\end{equation}
here the elements of the standard basis of $\mathbb{C}^{2^{g-1}(2^g+1)}$ are relabeled as
$\{u_{\varepsilon,\delta}\}_{(\varepsilon,\delta)\in E}$ (see \eqref{eE} for $E$).
For details on this map see the map $\theta^2$ in \cite{salvati-modular}.

We have $\mathbb{C}^{2^{g-1}(2^g+1)}\,\cong \,\operatorname{Sym}^2
\mathbb{C}^{2^g}$. It is well-known that $\theta_g^2\,=\, v\, \circ \, \Theta_g$ for the Veronese embedding 
$$v\ :\ \mathbb{P}(\C^{2^g}) \ \longrightarrow\ \mathbb{P}(\operatorname{Sym}^2 \C^{2^g}),$$ up to
an automorphism of $\mathbb{P}(\operatorname{Sym}^2 \C^{2^g})$. We will need 
a ``metric'' version of this result. For a complex vector space $V$ with an
inner product, denote by $\omega_{\mathbb{P}(V)}^{FS}$ the 
Fubini--Study metric on $\mathbb{P}(V)$. Consider the particular Veronese embedding 
$$v_2\ :\ \mathbb{P}(\C^n)\ \longrightarrow\ \mathbb{P}(\operatorname{Sym}^2 \C^n)$$ defined by
\begin{equation}\label{ev}
\sum_{i=1}^n \lambda_i e_i\,\ \longmapsto\,\ \sum_{i=1}^n \lambda_i^2 e_i\odot e_i+\sum_{1\leq i<j\leq n}
\sqrt{2} \lambda_i\lambda_j e_i\odot e_j.
\end{equation}

\begin{lemma}\label{lefs}
For the above map $v_2\,:\,\mathbb{P} (\C^n) \,\longrightarrow\, \mathbb{P} (\operatorname{Sym}^2 \C^n)$,
$$v_2^\ast \,\omega_{\mathbb{P}(\operatorname{Sym}^2\C^n)}^{FS}\ =\ 2\,\omega_{\mathbb{P}(\C^n)}^{FS}.$$
\end{lemma}

\begin{proof}
This will be proved assuming that $n\,=\,2$; proof of the general case is the same.
Identifying $\mathbb{P}(\operatorname{Sym}^2 \C^2)$ with $\mathbb{P}^2$, choose homogeneous coordinates
$u,\,v$ on $\mathbb{P}^1$ and $x,\,y,\,z$ on $\mathbb{P}^2$ such that
\begin{enumerate}
\item $v_2\,:\,\mathbb{P}^1 \,\longrightarrow\, \mathbb{P}^2$ is given by
$[u:\,v]\,\longmapsto\, [u^2:\,\sqrt{2}uv:\,v^2]\,=\,[x:\,y:\,z]$,

\item $\omega_{\mathbb{P}^2}^{FS}\,=\,\frac{\sqrt{-1}}{2} \partial \overline{\partial}\log
(|x|^2+|y|^2+|z|^2)$, and

\item $\omega_{\mathbb{P}^1}^{FS}\,=\,
\frac{\sqrt{-1}}{2}\partial \debar \log(|u|^2+|v|^2)$.
\end{enumerate}
Therefore, $$ v_2^\ast \,\omega_{\mathbb{P}^2}\,=\,\frac{\sqrt{-1}}{2} \partial \overline{\partial}
\log(|u^2|^2+|\sqrt{2}uv|^2+|v^2|^2)\,=\, \sqrt{-1} \partial \overline{\partial}
\log(|u|^2+|v|^2)\,=\,2\,\omega_{\mathbb{P}^1}.$$
This completes the proof.
\end{proof}

A notation for the next lemma: let $(\mathbb{Z}_2^g)^{(2)}$ denote 
the unordered pairs of elements of $\mathbb{Z}_2^g = (\mathbb{Z}/2\mathbb{Z})^g$. So
$(\mathbb{Z}_2^g)^{(2)}$ is the quotient of $\mathbb{Z}_2^g\times\mathbb{Z}_2^g $ by the equivalence relation
$(\alpha,\,\alpha')\,\sim\, (\alpha',\,\alpha)$ for $\alpha,\, \alpha' \,\in\, \mathbb{Z}_2^g$.

\begin{lemma}\label{lfs}
There exists a unitary automorphism $p$ of $\mathbb{P} (\operatorname{Sym}^2 \C^{2^g})
\ =\ \mathbb{P} (\C^{2^{g-1}(2^g+1)})$ such that 
\begin{equation}
p\circ \theta_g^2\ =\ v_2\circ \Theta_g,
\end{equation}
where the maps $\Theta_g$, $\theta^2_g$ and $v_2$ are constructed in \eqref{tau2},
\eqref{et2} and \eqref{ev} respectively.
\end{lemma}

\begin{proof}
Setting $\beta\,=\,z\,=\,x\,=\,0$ in \cite[Eq.~3]{gradients},
\begin{equation}\label{p1}
\theta_{\frac{\alpha}{2}}(0;\tau)\theta_{\frac{\alpha+\varepsilon}{2}}(0;\tau)
\,\,=\,\, \frac{1}{2^g}\sum_{\sigma
\in \mathbb{Z}_2^g} (-1)^{\alpha\cdot \sigma}\theta^2_{\varepsilon,\sigma}(0;\tau)\ \ \ 
\forall\ \, \, \alpha,\, \varepsilon\,\in\, \mathbb{Z}_2^g
\end{equation}
(in our notation); this is a special case of the Riemann Addition Formula \cite[p.~139]{Igusa}. 
Substituting $\alpha'\,=\,\alpha+\varepsilon$ in \eqref{p1} we get the following:
\begin{equation}\label{j1}
\theta_{\frac{\alpha}{2}}(0;\tau)\theta_{\frac{\alpha'}{2}}(0;\tau)\ =\
\frac{1}{2^g}\sum_{\sigma \in \mathbb{Z}_2^g} (-1)^{\alpha\cdot \sigma}\theta^2_{\alpha+\alpha',\sigma}(0;\tau)\ \ \ \forall\ \,\, 
\alpha,\,\alpha'\,\in\,\mathbb{Z}_2^g.
\end{equation}
For integers $m^1,\, m^2$ with $m^1\cdot m^2\,=\,1\,\mod \,2$, we have $$\theta\tcc{m^1/
2}{m^2/2}(0;\tau)\ \,=\ \,0$$ (see \cite[Chapter 1, Theorem 4]{Gu3}). Consequently, $\theta^2_{\alpha+\alpha',
\sigma}(0;\tau)\,=\,0$ whenever $(\alpha+\alpha')\cdot \sigma
\,\neq\, 0$ and hence
$$\frac{1}{2^g}\sum_{\sigma\in\mathbb{Z}_2^g} (-1)^{\alpha\cdot\sigma}\theta^2_{\alpha+\alpha',\sigma}(0;\tau)\ =\ 
\frac{1}{2^g}\sum_{\sigma\in\mathbb{Z}_2^g\,\vert\, (\alpha+\alpha',\sigma)\in E} (-1)^{\alpha\cdot\sigma}\theta^2_{\alpha+\alpha',\sigma}(0;\tau),$$
where $E$ is defined in \eqref{eE}. Therefore we can write
\begin{equation}\label{eq:emme1}
\sqrt{2}\theta_{\frac{\alpha}{2}}(0;\tau)\theta_{\frac{\alpha'}{2}}(0;\tau)\,\,=\,\,
\sum_{(\varepsilon,\sigma)\in E} M_{\alpha,\alpha';\varepsilon,\sigma}\theta^2_{\varepsilon,\sigma}(0;\tau)\ \ \ (\alpha,\, \alpha'\,\in\, \mathbb{Z}_2^g,\
\alpha\, \not=\, \alpha')
\end{equation}
\begin{equation}\label{eq:emme2}
\theta_{\frac{\alpha}{2}}(0;\tau)^2\,\,=\,\,\sum_{(\varepsilon,\sigma)\in E} M_{\alpha,\alpha;\varepsilon,\sigma}\theta^2_{\varepsilon,\sigma}(0;\tau)\
\ \ (\alpha\,\in\, \mathbb{Z}_2^g),
\end{equation}
where
\[
M_{\alpha,\alpha';\varepsilon,\sigma}\,\,\,=\,\,\,
\begin{cases}
\frac{1}{2^g}(-1)^{\alpha\cdot \sigma} & \,\text{if }\,\alpha+\alpha'\,=\,\varepsilon\,=\,0\\
\frac{\sqrt{2}}{2^g}(-1)^{\alpha\cdot \sigma} & \,\text{if }\,\alpha+\alpha'\,=\,\varepsilon\,\neq\, 0\\
0 &\, \text{otherwise.}
\end{cases}
\]
Here $M$ should be understood as a $\frac{4^g+2^g}{2}\times \frac{4^g+2^g}{2}$ matrix, with rows
indexed by (unordered) pairs $(\alpha,\,\alpha')\,\in\, (\mathbb{Z}_2^g)^{(2)}$ and columns indexed by
even characteristics $(\varepsilon,\,\delta)$. Clearly $M$ induces the
automorphism $p$ in the statement via the linear map
$$u_{\varepsilon,\delta}\,\,\longmapsto\,\, \sum_{(\alpha,\,\alpha')\,\in\, 
(\mathbb{Z}_2^g)^{(2)}} M_{\alpha,\alpha';\varepsilon,\delta} e_\alpha\odot e_{\alpha'}$$ 
(this is the content of equations \eqref{eq:emme1} and \eqref{eq:emme2}). Since $M$ is a real matrix, it remains to show that $M$ is 
a scalar multiple of an orthogonal matrix. 

Computing
$$(M^t\cdot M)_{\varepsilon,\sigma;\varepsilon',\sigma'}\ =\ \sum_{\alpha,\alpha'\in (\mathbb{Z}_2^g)^{(2)}}
M_{\alpha,\alpha';\varepsilon,\sigma}M_{\alpha,\alpha';\varepsilon',\sigma'}$$
we see that it is zero whenever $\varepsilon\,\neq\, \varepsilon'$, because
$M_{\alpha,\alpha';\varepsilon,\sigma}\,\neq\, 0$ if and only if $\alpha'\,=\,\alpha+\varepsilon$.
 
For the same reason, when $\varepsilon\,=\,\varepsilon'$, we have
$$(M^t\cdot M)_{\varepsilon,\sigma;\varepsilon,\sigma'}\ =\ \sum_{\alpha,\alpha'\in (\mathbb{Z}_2^g)^{(2)}\vert\alpha+\alpha'
\ =\ \varepsilon}M_{\alpha,\alpha';\varepsilon,\sigma}M_{\alpha,\alpha';\varepsilon,\sigma'}.$$
If $\varepsilon\,=\,\varepsilon'\,=\,0$, then
$$(M^t\cdot M)_{0,\sigma;0,\sigma'}\,\,=\,\,\sum_{(\alpha,\alpha')\in (\Z_2^g)^{(2)}\vert\alpha=\alpha'}
M_{\alpha,\alpha';0.\sigma}M_{\alpha,\alpha';0.\sigma'}
$$
$$
=\,\,\sum_{\alpha\in \Z_2^g} M_{\alpha,
\alpha;0,\sigma}M_{\alpha,\alpha,0,\sigma'}
\,\, =\, \,\frac{1}{4^g}\sum_{\alpha\in \mathbb{Z}_2^g}(-1)^{\alpha\cdot (\sigma+\sigma')},$$
and this is zero when $\sigma+\sigma'\,\neq\, 0$ (which holds if and only
if $\sigma\,\neq\,\sigma'$), and it is $\frac{1}{2^g}$ otherwise.

If $\varepsilon\,=\,\varepsilon'\,\neq\, 0$, the pair $(\alf, \,\alf +\epsilon)$ is the same as the pair $(\alf - \epsilon ,\,\alfa)$ in $(\mathbb{Z}_2^g)^{(2)}$. Hence we have
$$(M^t\cdot M)_{\varepsilon,\sigma;\varepsilon,\sigma'}\,\,=\,\,\sum_{(\alpha,\alpha')\in (\Z_2^g)^{(2)}\vert\alpha=\alpha'+\varepsilon}
M_{\alpha,\alpha';\varepsilon.\sigma}M_{\alpha,\alpha';\varepsilon.\sigma'}
$$
$$
=\,\,\frac{1}{2}\sum_{\alpha\in \Z_2^g} M_{\alpha,\alpha+\varepsilon;\varepsilon,\sigma}
M_{\alpha,\alpha+\varepsilon,\varepsilon,\sigma'}
\,\,=\,\, \frac{1}{4^g}\sum_{\alpha\in \mathbb{Z}_2^g}(-1)^{\alpha\cdot (\sigma+\sigma')}$$
and, again, this is $\frac{1}{2^g}\delta_{\sigma \sigma'}$. Combining these it follows that
$$(M^t\cdot M)_{\varepsilon,\sigma;\varepsilon',\sigma'}\ =\ \begin{cases}
0 & \text{if } (\varepsilon,\, \sigma)\,\neq\, (\varepsilon',\,\sigma')\\
\frac{1}{2^g} & \text{if } (\varepsilon,\, \sigma)\,=\, (\varepsilon',\,\sigma'), \end{cases}$$
thus $M$ is constant scalar multiple of an orthogonal matrix.
\end{proof}

\begin{proposition} \label{prop:descends}
The form $\Theta_g^\ast \omega_{FS}$ descends to $\A_g$, where $\omega_{FS}$ is the form in \eqref{eq:omegafs} and $\Theta_g$ is
the map in \eqref{tau2}.
\end{proposition}

\begin{proof}
With an abuse of notation, we will denote the Fubini-Study metric on every $\mathbb{P}^n$ by
the same symbol $\omega_{FS}$.
Both $v_2$ and $p$ in Lemma \ref{lfs} preserve $\omega_{FS}$ up to a nonzero scalar multiple.
Therefore,
$$\Theta_g^\ast \omega_{FS}\,=\,\Theta_g^\ast v_2^\ast\omega_{FS}\,=\,
(\theta^2)^\ast p^\ast \omega_{FS}\,=\,(\theta^2)^\ast \omega_{FS}$$ up to a
nonzero scalar multiple. Consequently, in order to prove that $\Theta_g^\ast \omega_{FS}$ descends to
$\A_g$ it suffices to show that $(\theta^2)^\ast \omega_{FS}$ does so. The reason behind this approach
is that $\theta^2$ is more symmetric with respect to the Theta Transformation Formula given below
in \eqref{ttf}.

Take any matrix $$M\,\,=\,\,\begin{pmatrix}a & b\\c & d\end{pmatrix}\,\,\in\,\, {\rm Sp}(2g,\mathbb{Z}).$$ Recall
the Theta transformation formula \cite[Proposition 8.6.1]{lange-birkenhake}: for any $c^1,\,c^2\,\in\,
\mathbb{R}^g$, there exist $k(M,\,\varepsilon/2,\, \delta/2)\,\in\, \mathbb R$ and $\kappa(M)\,\in\,
S^1\subset \mathbb{C}$ such that
\begin{multline}
\label{ttf}
\theta\tcc{M(c^1,c^2)^1}{M(c^1,c^2)^2}(
^t (c\tau + d )^{-1} z; M(\tau))=\\
= \det(c \tau + d)^{1/2}\kappa(M)
\exp({\pi\sqrt{-1} k(M,c^1,c^2) + \pi\sqrt{-1} z^t
(c\tau + d )^{-1} cz })\theta\tcc{c^1}{c^2}(z;\tau),
\end{multline}
where $M(\varepsilon/2,\delta/2)^1$ and $M(\varepsilon/2,\delta/2)^2$ are as in \cite[Lemma 8.4.1(b)]{lange-birkenhake}.

Taking the absolute value of the square of \eqref{ttf} at $z\,=\,0$, 
$$\left|\theta^2\tcc{M(c^1,c^2)^1}{M(c^1,c^2)^2}(0;M(\tau))\right|\ =\ |\det (c\tau+d)|\cdot \left|\theta^2\tcc{c^1}{c^2}(0;\tau)\right|;$$
substituting $c^1\,=\,m^1/2$, $c^2\,=\,m^2/2$, where $m^1,\, m^2\, \in\, {\mathbb Z}^g$, this gives the following:
\begin{equation}\label{eq:chchange3}
\left|\theta^2\tcc{M(m^1/2,m^2/2)^1}{M(m^1/2,m^2/2)^2}(0;M(\tau))\right|\,=\,|\det (c\tau+d)|\cdot |\theta^2_{[m_1],[m_2]}(0;\tau)|\ \ \ \forall\ \,\, m^1,\,m^2
\,\in\,\mathbb{Z}^g.
\end{equation}

We now make the following claim:

\textbf{Claim:} Let $m^1,\,m^2\,\in\,{\mathbb Z}^g$. Then
$$2 M(m^1/2,m^2/2)^1\, \in\, {\mathbb Z}^{2g}\ \ \, \text{and }\ \ \, 2 M(m^1/2,m^2/2)^2\, \in\, {\mathbb Z}^{2g}.$$
Moreover, the map
\begin{equation}\label{eq:chchange}
T\,\,:\,\,E\,\lra\, E,\, \ \ \ ([m^1],\, [m^2])\,\,\longmapsto\,\, ([2 M(m^1/2,\,m^2/2)^1],\,[2 M(m^1/2,\,m^2/2)^2]),
\end{equation}
is a bijection.

The proof of the above Claim is given after the proof of the proposition.

In view of the above Claim, \eqref{eq:chchange3} can now be written as
\begin{equation}\label{eq:chchange2}
|\theta^2_{T([m^1],[m^2])}(0;M(\tau))|\,=\,|\det (c\tau+d)|\cdot |\theta^2_{[m_1],[m_2]}(0;\tau)|\ \ \ \forall\,\,\ m^1,\,m^2\,\in\,\mathbb{Z}^g.
\end{equation}

Let
$$W(\tau):\,\, =\,\, \sum_{(\varepsilon,\delta)\in E} \left|\theta^2_{\varepsilon,\delta}(0;\tau)\right|^2;$$
then $(\theta^2)^\ast \omega_{FS}\,=\,\frac{\sqrt{-1}}{2}\partial\debar \log W$. The proof of the proposition
will be completed by showing that the form $\partial\debar \log W$ descends to $\A_g$.

Using the Claim and equation \eqref{eq:chchange2},
\begin{gather*}
W(M(\tau))\,=\,\sum_{(\varepsilon,\delta)\in E} \left|\theta^2_{(\varepsilon,\delta)}(0;M(\tau))\right|^2\,=\,
\sum_{(\varepsilon,\delta)\in E} \left|\theta^2_{T(\varepsilon,\delta)}(0;M(\tau))\right|^2\\
 =\ |\det(c\tau+d)|^2\cdot \sum_{(\varepsilon,\delta)\in E} \left|\theta^2_{(\varepsilon,\delta)}(0;\tau)\right|^2\,=\,|\det(c\tau+d)|^2\cdot W(\tau).
\end{gather*}
But then 
$$\partial\debar \log W(M(\tau))\,=\,\partial\debar \log\det(c\tau+d)+\partial\debar\log\overline{\det(c \tau+d)}+\partial\debar \log W(\tau)
\,=\,\partial\debar \log W(\tau).$$
Since $M\,\in\, \operatorname{Sp}(2g,\mathbb{Z})$ was arbitrary, this proves that $(\theta^2)^\ast\omega_{FS}$
descends to $$\mathbb{H}_g\backslash \operatorname{Sp}(2g,\mathbb{Z})
\,\ =\, \A_g .$$ As noted before, this completes the proof assuming the Claim.
\end{proof}

\begin{proof}[{Proof of the Claim}]
For any $c^1$ and $c^2$ in $\mathbb{R}^g$ it is well-known that $\theta\tcc{c^1}{c^2}(z;\tau)$ is even (respectively, odd) in $z$ if and only if $c^1,\,c^2\,\in\, \frac{1}{2}{\mathbb Z}^g$
and $(2 c^1)\cdot (2c^2)\,=\,0\mod\, 2$ (respectively, $1\mod \,2$); see
\cite[Chapter 1, Th. 4]{Gu3}.

Take $m^1,\,m^2\,\in\, \mathbb{Z}^g$ with $m^1\cdot m^2\ =\ 0\mod\, 2$. Then $\theta\tcc{m^1/2}{m^2/2}(z;\tau)$ is even in $z$, and therefore
$\theta\tcc{M(m^1/2,m^2/2)^1}{M(m^1/2,m^2/2)^2}(v;M(\tau))$ is even in $v$ by \eqref{ttf}. This implies that $$2M(m^1/2,m^2/2)^1,\,\, 2M(m^1/2,m^2/2)^1
\,\,\in\,\, {\mathbb Z}^g,$$
and $(2M(m^1/2,m^2/2)^1)\cdot (2M(m^1/2,m^2/2)^2)\,=\,0\, \mod\, 2$. It now follows that $$([2M(m^1/2,m^2/2)^1],\, [2M(m^1/2,m^2/2)^2])\,\in\, E.$$

One can show that
\begin{equation}\label{eq:injectivityandgooddef}
\begin{cases}
[2M(m^1/2,m^2/2)^1]\,=\,[2M(n^1/2,n^2/2)^1]\\
[2M(m^1/2,m^2/2)^2]\,=\,[2M(n^1/2,n^2/2)^2]
\end{cases}\ \iff\ \begin{cases}
 [m^1]\,=\,[n^1]\\
[m^2]\,=\,[n^2].
\end{cases}
\end{equation}
{}From \eqref{eq:injectivityandgooddef} it follows that the map $T$ in \eqref{eq:chchange} is well defined and injective.

We will prove the ``$\implies$" implication in \eqref{eq:injectivityandgooddef}; the proof of the converse is similar.
By \eqref{eq:subtle}, if $[2M(m^1/2,m^2/2)^1]\,=\,[2M(n^1/2,n^2/2)^1]$ and $[2M(m^1/2,m^2/2)^2]\,=\,[2M(n^1/2,n^2/2)^2]$, then we have
$$\theta\tcc{M(m^1/2,m^2/2)^1}{M(m^1/2,m^2/2)^2}(v, M(\tau))\,\,=\,\,\pm \theta\tcc{M(n^1/2,n^2/2)^1}{M(n^1/2,n^2/2)^2}(v, M(\tau)),$$
and then by \eqref{ttf},
\begin{gather*}
\det(c\tau+d)^{1/2}\kappa(M)\exp(\pi \sqrt{-1}k(M,m^1/2,m^2/2)+\pi \sqrt{-1} z^t (c\tau+d)^{-1} cz)\theta\tcc{m^1/2}{m^2/2}(z;\tau)\,=\\
=\,\pm\det(c\tau+d)^{1/2}\kappa(M)\exp(\pi \sqrt{-1}k(M,n^1/2,n^2/2)+\pi \sqrt{-1} z^t (c\tau+d)^{-1} cz)\theta\tcc{n^1/2}{n^2/2}(z;\tau).
\end{gather*}
{}From this it follows that $$\theta\tcc{m^1/2}{m^2/2}(z;\tau)\ =\ \lambda\theta\tcc{n^1/2}{n^2/2}(z;\tau)$$
for some nonzero $\lambda\,\in\, \mathbb{C}$. By \cite[p. 173, Lemma 1]{Igusa}, this implies that $[m^1]\,=\,[n^1]$ and $[m^2]\,=\,[n^2]$.
\end{proof}

\section{The Theta projective structure}\label{sec:theta-proj-struct}

In this section we will compute $\debar\beta^\theta$ explicitly (see \eqref{pr-strs} and \eqref{k2}). One can work
with coordinates using the following standard procedure (see the example of \cite{Ma}): Let $$\pi\ :\ \mathcal{C}\
\longrightarrow\ B$$ be a holomorphic family of compact Riemann surfaces of genus $g$, as in \eqref{en1},
with $B$ being simply connected. Then $R^1\pi_\ast \mathbb{Z}$ is
a constant sheaf. In other words, there are constant sections $\{a_i,\, b_i\}_{i=1}^g$
of $R^1\pi_\ast \mathbb{Z}$ such that
$\{a_i(b),\, b_i(b)\}_{i=1}^g$ is a basis of $H_1(C_{b},\, {\mathbb Z})$ for any $b\, \in\, B$
satisfying the following conditions:
$$a_i(b)\cap b_j(b)\ =\ \delta_{ij},\ \ \ a_i(b)\cap a_j(b)\ =\ 0,\ \ \ b_i(b)\cap b_j(b)\ =\ 0\ \ \ \forall\ \, i,\,j.$$
Such a basis is also called a marking for $C_b$.
We also have unique holomorphic sections $\{\omega_i\}_{i=1}^g$ of $\pi_\ast \mathcal{K}$ (see \eqref{em}) for which
\begin{equation}\label{eo}
\int_{a_i(b)}\omega_j(b)\ =\ \delta_{ij},\ \ \, b\, \in\, B,
\end{equation}
and in this way one constructs a holomorphic period map
\begin{equation}\label{tau}
j\,\,:\,\, B\, \,\longrightarrow\,\, \mathbb{H}_g,\ \ \ j(b)_{ij}\, :=\, \int_{b_i(b)}\omega_j(b)
\end{equation}
for all $b\, \in\, B$.

Consider the sections $\sigma$ and $\eta$ of $\Pi_\ast \mathbb{L}$ defined in \eqref{eta}, \eqref{sigma}, where $\Pi$
and $\mathbb{L}$ are constructed in \eqref{en4} and \eqref{en4b} respectively.
Recall from \cite[Propositions 6.3 and 6.7]{BGV} that with respect to this marking, the
sections $\sigma$ and $\eta$ of $\Pi_\ast \mathbb{L}$ can be expressed as 
\begin{equation}
\label{eq:sigma}
\sigma(b)\,=\, \Omega_b+\frac{1}{2\sum_{u\in U} |\theta_u(0;j(b))|^2 }\sum_{i,j=1}^g\sum_{u\in U} \overline{\theta_u(0;j(b))}\frac{\partial^2\theta_u}{\partial z_i\partial z_j}(0;j(b))p_1^\ast \omega_i\wedge p_2^\ast \omega_j,
\end{equation}
\begin{equation}
\label{eq:eta}
\eta(b)\ =\ \Omega_b-\pi \sum_{i,j=1}^g (\operatorname{Im} j(b))^{ij} p_1^\ast \omega_i
\wedge p_2^\ast \omega_j ,
\end{equation}
where $\Omega_b\in H^0(S_b,\, K_{S_b}(2\Delta_b))$ is the Canonical Bidifferential in \cite{Gu2}; see
also \cite[p.~20]{Fay}. The expression for $\Omega$ in our notation is given in \cite[(5.2)]{BGV}.

\begin{remark}
We note that $\Omega$ is a holomorphic section of $\Pi_\ast\mathbb{L}$
over $B$. This can be seen immediately by examining the expression of $\Omega$ in \cite[p.~20]{Fay}.
Consequently, the section $\mathbf{r}(\Omega)$ (which turns out to be a section of $\widehat{\mathcal{V}}$)
is also holomorphic, where $\mathbf{r}$ is the map in \eqref{eqr} (see also \cite[p.~300]{ZT}); the projective
structure corresponding to $\mathbf{r}(\Omega_b)$ is denoted there by $R_{X_b}$. It should be clarified that
$\Omega$ does not exist globally on $\Mg$; its construction depends on the choice of a symplectic basis.
Indeed, the Canonical Bidifferential $\Omega_b$ is not canonical; for example, it is not fixed by the
action of the automorphism group $\operatorname{Aut}(C_b)$.
\end{remark}

\begin{lemma}
\label{lemma:multiplication}
The following diagram of maps is commutative:
$$\begin{tikzcd}
& & \pi_\ast \mathcal{K}^{\otimes 2} \arrow[d] \\
\Pi_\ast ({\widetilde p}^*{\mathcal K}\otimes {\widetilde q}^*{\mathcal K}) \arrow[r] \arrow[rru, "\mathbf{m}"] & \Pi_\ast \mathbb{L} \arrow[r, "\mathbf{r}"] & \mathcal{V}
\end{tikzcd}$$
where $\mathbf{m}$ is given fiber-wise by the natural multiplication map 
$$\mathbf{m}_b\,:\, H^0(S_b,\, p_1^\ast K_{C_b}\otimes p_2^\ast K_{C_b})\,= \,H^0(C_b,\, K_{C_b})^{\otimes 2}\,
\longrightarrow\, H^0(C_b,\, 2 K_{C_b}),$$ and the vertical map on the right is the inclusion map
of the kernel of $\Psi$ in \eqref{psi}; the map $\mathbf r$ is constructed in \eqref{eqr}.
\end{lemma}

\begin{proof}
Consider the following commutative diagram --- with exact rows --- of coherent sheaves on ${\mathcal C}\times_B{\mathcal C}$:
$$\begin{tikzcd}
0 \arrow[r] & {\widetilde p}^*{\mathcal K}\otimes {\widetilde q}^*{\mathcal K} \otimes
{\mathcal O}_{{\mathcal C}\times_B{\mathcal C}}(-{\bf\Delta}_B) \arrow[d] \arrow[r] & \mathbb{L} \arrow[r, "\mathbf{r'}"] \arrow[d, "\operatorname{Id}"] & {\mathbb L}/({\mathcal O}_{{\mathcal C}\times_B
{\mathcal C}}(-3{\bf\Delta}_B)\otimes {\mathbb L}) \arrow[r] \arrow[d, "\Psi'"] & 0 \\
0 \arrow[r] & {\widetilde p}^*{\mathcal K}\otimes {\widetilde q}^*{\mathcal K} \arrow[r] & \mathbb{L} \arrow[r] & {\mathbb L}/({\mathcal O}_{{\mathcal C}\times_B
{\mathcal C}}(-2{\bf\Delta}_B)\otimes {\mathbb L}) \arrow[r] & 0.
\end{tikzcd}$$
Note that by construction, $\mathbf{r}\,=\,\Pi_\ast \mathbf{r'}$ and $\Psi\,=\,\Pi_\ast \Psi'$ (see \eqref{eqr}). By the snake
lemma,
$$\ker \Psi'\,\cong\, ({\widetilde p}^*{\mathcal K}\otimes {\widetilde q}^*{\mathcal K})/
({\widetilde p}^*{\mathcal K}\otimes {\widetilde q}^*{\mathcal K} \otimes
{\mathcal O}_{{\mathcal C}\times_B{\mathcal C}}(-{\bf\Delta}_B))\,=\,
({\widetilde p}^*{\mathcal K}\otimes {\widetilde q}^*{\mathcal K})\big\vert_{{\bf\Delta}_B}.$$
Using this identification we have the diagram
\begin{equation}\label{eq:bigd}
\begin{tikzcd}
&& ({\widetilde p}^*{\mathcal K}\otimes {\widetilde q}^*{\mathcal K})|_{{\bf\Delta}_B}\arrow[d] \\
{\widetilde p}^*{\mathcal K}\otimes {\widetilde q}^*{\mathcal K} \arrow[r] \arrow[rru, "\mathbf{m'}"]
& \mathbb{L} \arrow[r, "\mathbf{r'}"] & \mathbb{L}/(\mathbb{L}\otimes
{\mathcal O}_{{\mathcal C}\times_B{\mathcal C}}(-3{\bf\Delta}_B)),
\end{tikzcd}
\end{equation}
where the vertical map on the right is the inclusion map of the kernel of $\Psi'$ while $\mathbf{m'}$
is the natural projection. Taking push-forward, using $\Pi$, of the diagram in \eqref{eq:bigd} we obtain the following:
\begin{equation}\label{eq:bigd2}
\begin{tikzcd}
&&\Pi_\ast ({\widetilde p}^*{\mathcal K}\otimes {\widetilde q}^*{\mathcal K})|_{{\bf\Delta}_B})\arrow[d] \\
\Pi_\ast({\widetilde p}^*{\mathcal K}\otimes {\widetilde q}^*{\mathcal K}) \arrow[r] \arrow[rru, "\Pi_\ast\mathbf{m'}"]
& \Pi_\ast \mathbb{L} \arrow[r, "\mathbf{r}"] & \mathcal{V};
\end{tikzcd}
\end{equation}
the vertical map on the right is the kernel of $\Pi_\ast \Psi'\,=\,\Psi$.

Now consider the restriction map $d_b\,:\,H^0(K_{S_b})\, \longrightarrow\, H^0(K_{S_b}|_{\Delta_b})$ and the multiplication map
$$\mathbf{m}_b\,:\,H^0(S_b,\, K_{S_b})\,=\, H^0(C_b,\, K_{C_b})^{\otimes 2}
\, \longrightarrow\, H^0(C_b,\, 2 K_{C_b}).$$ These coincide respectively with the restrictions of $\Pi_\ast \mathbf{m'}$ and $\mathbf{m}$
to $b$. There is a canonical isomorphism $i_b\,:\,H^0(K_{S_b}|_{\Delta_b})\,\longrightarrow\, H^0( K_{C_b}^{\otimes 2})$
for which ${\textbf m}_b\,=\,i_b\circ d_b$. Fiber-wise, the isomorphisms $i_b$ give an isomorphism
$$\mathbf{i}\ :\ \Pi_\ast ((\widetilde p^*{\mathcal K}\otimes
{\widetilde q}^*{\mathcal K})|_{{\bf\Delta}_B})\ \longrightarrow\ \pi_\ast \mathcal{K}^{\otimes 2}$$
fitting in the commutative diagram
\begin{equation}
\label{eq:theiso}
\begin{tikzcd}
& \Pi_\ast(\widetilde p^*{\mathcal K}\otimes {\widetilde q}^*{\mathcal K}) \arrow[ld, "\Pi_\ast\mathbf{m'}"'] \arrow[rd, "\mathbf{m}"] & \\
\Pi_\ast ((\widetilde p^*{\mathcal K}\otimes {\widetilde q}^*{\mathcal K})|_{{\bf\Delta}_B}) \arrow[rr, "\mathbf{i}"] && \pi_\ast \mathcal{K}^{\otimes 2}.
\end{tikzcd}
\end{equation}
Take the morphism $\mathbf{m}$ in the statement of the lemma to be the composition of maps
$\mathbf{i}\circ\Pi_\ast\mathbf{m'}$ (see \eqref{eq:bigd2} and \eqref{eq:theiso} for these maps). Now the
lemma follows from \eqref{eq:bigd2} and \eqref{eq:theiso}.
\end{proof}

Consider the Kodaira--Spencer deformation map for the family $\pi$
\begin{equation}\label{eks}
\rho_b\ :\ T_b B\ \longrightarrow\ H^1(C_b,\, T_{C_b})
\end{equation}
and also the Serre duality $s_b\,:\, H^0(C_b,\, 2K_{C_b})\,\stackrel{\sim}{\longrightarrow}\, 
H^1(C_b,\, T_{C_b})^\ast$ given by the pairing
\begin{equation}\label{eq:serre}
H^1(C_b,\, T_{C_b})\otimes H^0(C_b,\, 2 K_{C_b})\, \longrightarrow\, \mathbb{C},\ \ \
\xi \otimes \omega\,\longmapsto\, \int_{C_b} \xi\cup \omega.
\end{equation}
Let
$$
\lambda_b\ :\ H^0(C_b,\, 2K_{C_b})\ \longrightarrow\ T^\ast_b B
$$
be the composition of maps
$$H^0(C_b,\, 2K_{C_b})\, \xrightarrow{\,\,\,s_b\,\,\,}\, H^1(T_{C_b})^\ast
\, \xrightarrow{\,\,\,(\rho_b)^\ast\,\,\,}\, T^\ast_b B,$$
where $\rho_b$ and $s_b$ are the maps in \eqref{eks} and \eqref{eq:serre} respectively. It produces
a morphism $$\lambda\ :\ \pi_\ast\left(\mathcal{K}^{\otimes 2}\right)\ \longrightarrow\ T^\ast B.$$

\begin{lemma}
\label{lemma:lambda}
The following holds:
$$
j^\ast d \tau_{ij}\ =\ \lambda(-\omega_{i}\omega_j),
$$
where $j$ is the Torelli map in \eqref{tau} and $\omega_i$, $1\, \leq\, i\, \leq\, g$, are as in \eqref{eo}.
\end{lemma}

\begin{proof}
Fix a point $b_0\,\in\, B$ and take a tangent vector $v\,\in\, T_{b_0} B$. Let $C\,=\,C_{b_0}$ and denote
$\omega_i\,=\,\omega_i(b_0)$, $a_i\,=\, a_i(b_0)$ and $b_i\,=\, b_i(b_0)$ for $1\, \leq\, i\, \leq\, g$. By
definition,
$$\lambda(-\omega_i\omega_j)(v)\ =\ \int_C \rho(v)\cup (-\omega_i\omega_j).$$

We need to prove the following:
\begin{equation}\label{eq:dedot}
j^\ast d\tau_{ij}(v)\ =\ \int_{C} \rho(v)\cup (-\omega_i\omega_j).
\end{equation}
This is quite a standard computation (see \cite[Ch.~XI, \S~8]{ACGH}); we give the details for convenience.
Fix a basis $\{\alpha_1,\, \cdots ,\, \alpha_g,\, \beta_1,\, \cdots ,\, \beta_g\}$ of $H^1(C,\, \C)$ dual to the marking, so we have
$$\int_{a_i}\alpha_j\,=\,\delta_{ij},\qquad \int_{a_i}\beta_j\,=\,0,\qquad \int_{b_i}\alpha_j\,=\,0,\qquad \int_{b_i}\beta_j\,=\,\delta_{ij}.$$
Define the map to the Grassmannian parametrizing $g$ dimensional subspaces in $H^1(C,\,\C)$
\begin{equation}
H\,:\, B\,\longrightarrow \,\operatorname{Grass}(g,\, H^1(C,\C)),\ \ \, b\,\longmapsto\, \operatorname{span}\left(\alpha_i+ \sum_{h=1}^g j(b)_{ih}\beta_h\right)_{i=1,\cdots ,g}.
\end{equation}
Comparing the integrals on the $a_i$ and $b_i$, one easily sees that
\begin{equation}
\label{eq:acg1}
\omega_i\ =\ \alpha_i+ \sum_{h=1}^g j(b_0)_{ih}\beta_h
\end{equation}
as elements of $H^1(C,\,\C)$, and therefore we have $H(b_0)\,=\,H^{1,0}(C_{b_0})$.
Then we can identify 
\begin{equation}
\label{eq:grasstangent}
{T_{H(b_0)}\operatorname{Grass}(g,\,H^1(C,\C))} \ \cong\ \operatorname{Hom}\left(H^{1,0}(C),\, H^1(C,
\C)/H^{1,0}(C)\right)\ \cong\ H^{1,0}(C)^\ast\otimes H^{0,1}(C).
\end{equation}
Let $\{\omega_1^\ast,\,\cdots ,\,\omega_g^\ast\}$ denote the dual basis of $H^0(C,\, K_C)^\ast\,=\,H^{1,0}(C)^\ast$. Then through the identification in \eqref{eq:grasstangent}
we have, considering equation \eqref{eq:acg1}, $$d H (v)\ =\ \sum_i \omega_i^\ast\otimes \left[\sum_h j^\ast d\tau_{ih}(v) \beta_h\right],$$
where $[\cdot]$ denotes the class in $H^{0,1}(C)\,=\,H^1(C,\C)/H^{1,0}(C)$. On the other hand, we have
$$d H(v)\ =\ \sum_i \omega_i^\ast\otimes (\rho(v)\cup \omega_i)$$
(see \cite[p. 220 - 223]{ACGH}).
Then $\rho(v)\cup \omega_i\,=\, \left[\sum_h j^\ast d\tau_{ih}(v) \beta_h\right]$, and $$\int_C \rho(v)\cup (-\omega_i\omega_j)
\,=\,-\int_C (\rho(v)\cup \omega_i)\wedge \omega_j\,=\,-\sum_h j^\ast d\tau_{ih}(v)\int_C \beta_h\wedge \omega_j\,=\,j^\ast d\tau_{ij}(v).$$
This proves \eqref{eq:dedot}, and the proof of the lemma is completed.
\end{proof}

\begin{remark}
The pairing in \eqref{eq:serre} is unique up to a sign.
The choice of the pairing in \eqref{eq:serre} is consistent with \cite[(2.2)]{ZT}. This aligns
the scalar multiples of this work with those of \cite{BGT} (compare Proposition \ref{prop:debareta} below and
\cite[Theorem 5.6]{BGT}). The choice of
the opposite sign in Serre duality yields of course the opposite identification (see \cite[p.~58]{Ma};
the pairing is defined on p.~56 of \cite{Ma}).
\end{remark}

\begin{theorem}\label{prop:debarisfs}
Set
\begin{equation}\label{eq:wudoppio}
w(\tau)\ :=\ \sum_{u\in U} |\theta_u(0;\,\tau)|^2.
\end{equation}
On $B$,
\begin{equation}\label{eq:betatheta}
\beta^\theta(b)\ =\ (\Omega_b)|_{3\Delta}+4\pi\sqrt{-1}\sum_{1\leq j\leq i\leq g}
\frac{\partial \log w}{\partial \tau_{ij}}(j(b))\omega_i(b)\omega_j(b);
\end{equation}
here $\pi_\ast\mathcal{K}^{\otimes 2}\,\subset\, \mathcal{V}$ is the kernel of $\Psi$ in
\eqref{psi}.

If $\pi_\ast\mathcal{K}^{\otimes 2}$ is identified with $T^\ast B$ by the mapping
$-\omega_i\omega_j\ \longmapsto\ j^\ast (d\tau_{ij})$, then
\begin{equation}\label{eq:debarbetatheta}
\overline{\partial}\beta^\theta\ =\ 8\pi j^\ast \Theta_g^\ast \omega_{FS}
\end{equation}
(see \eqref{pr-strs} and \eqref{k2}), where $\omega_{FS}$ is the
Fubini-Study metric on $\mathbb{P}^{2^g-1}$ defined in \eqref{eq:omegafs}.
\end{theorem}

\begin{proof}
On $\mathbb H_g$, consider coordinates $\tau_{ij}$, with
$i \,\geq\, j$; we have $\tau_{ji}\,=\,\tau_{ij}$. Then the heat equation for second order theta
functions holds (see \cite[p. 20]{vg}):
\begin{gather*}
4\pi \sqrt{-1} (1+\delta_{ij}) \frac{\partial \theta_u}{\partial
\tau_{ij}} \,\,=\,\, \frac{\partial^2 \theta_u}{\partial z_i
\partial z_j}.
\end{gather*}
Expanding \eqref{eq:sigma} we have the following (notice the heat equation is used in the third equality):
\begin{gather*}
\sigma(b) \ =\ \Omega_b+\frac{1}{2 w(j(b))} \sum_{i,j=1}^g
\overline{\theta_u(0;\,j(b))}\, \frac{\partial^2 \theta_u}{\partial
z_i\partial z_j }(0;j(b)) p_1^\ast \omega_i\wedge p_2^\ast
\omega_j\\
=\ \Omega_b+\frac{1}{2w(j(b))} \sum_{i=1}^g \overline{\theta_u(0;\,j(b))}\,
\frac{\partial^2 \theta_u}{\partial z^2_i }(0;j(b)) p_1^\ast
\omega_i\wedge p_2^\ast \omega_i +
\\
+\,\, \frac{1}{2w(j(b))} \sum_{i >j} \overline{\theta_u(0;\,j(b))}\,
\frac{\partial^2 \theta_u}{\partial z_i\partial z_j }(0;j(b)) \left
( p_1^\ast \omega_i\wedge p_2^\ast \omega_j + p_1^\ast
\omega_j\wedge p_2^\ast \omega_i\right )\\
=\ \Omega_b+ \frac{4 \pi \sqrt{-1}}{w(j(b))} \sum_{i=1}^g
\overline{\theta_u(0;\,j(b))}\, \frac{\partial \theta_u}{\partial
\tau_{ii} }(0;j(b)) p_1^\ast \omega_i\wedge p_2^\ast \omega_i +
\\
+\,\, \frac{2 \pi \sqrt{-1}}{w(j(b))} \sum_{i >j}
\overline{\theta_u(0;\,j(b))}\, \frac{\partial \theta_u}{\partial
\tau_{ij} }(0;j(b)) \left ( p_1^\ast \omega_i\wedge p_2^\ast
\omega_j + p_1^\ast \omega_j\wedge p_2^\ast
\omega_i \right )\\
=\ \Omega_b+ \frac{2 \pi \sqrt{-1} }{w(j(b))} \sum_{i \geq j}
\overline{\theta_u(0;\,j(b))}\, \frac{\partial \theta_u}{\partial
\tau_{ij} }(0;j(b)) \left ( p_1^\ast \omega_i\wedge p_2^\ast
\omega_j + p_1^\ast \omega_j\wedge p_2^\ast \omega_i \right )\\
=\ \Omega_b + \frac{2\pi\sqrt{-1}}{w(j(b))} \sum_{i\geq j} \frac{\partial w}{\partial \tau_{ij}}(j(b))
\left ( p_1^\ast \omega_i\wedge p_2^\ast
\omega_j + p_1^\ast \omega_j\wedge p_2^\ast \omega_i \right ).
\end{gather*}

Consider $\mathbf{r}$ in \eqref{eqr}.
By definition, $\beta^\theta\,=\,\mathbf{r}(\sigma)$ (see \eqref{pr-strs}).
We have $$\mathbf{r}(\Omega)(b)\ =\ (\Omega_b)\big\vert_{3\Delta}$$
and the section $$\mathbf{r}(p_1^\ast \omega_i\wedge p_2^\ast \omega_j)\ =\
\mathbf{m}(p_1^\ast \omega_i\wedge p_2^\ast \omega_j)$$
sends any $b\, \in\, B$ to $\omega_i(b)\omega_j(b)$
by Lemma \ref{lemma:multiplication}. Thus \eqref{eq:betatheta} is proved.

Making the identification $j^* ( d \tau_{ij})\, =\, - \omega_i \cdot\omega_j$ we get 
$$\beta^\theta(b)\, =\,\sigma(b)|_{3\Delta}\,=\,
\Omega_b|_{3\Delta}-4\pi \sqrt{-1} \sum_{i\geq j} \frac{\partial \log w}{\partial \tau_{ij}}(j(b))
\cdot j^\ast (d \tau_{ij})(b).$$
Then 
$$\beta^\theta\ =\ \mathbf{r}(\Omega)-4\pi \sqrt{-1} j^\ast(\partial \log w);$$
here we, again, identify $\Omega^1_B\,\subset\, \mathcal{V}$ with the kernel of $\Psi$. Since $\debar
\Omega\,=\,0$ while $\mathbf{r}$ and $j$ are holomorphic, it follows that
\begin{equation}\label{ej}
\overline{\partial} \beta^\theta\,\, =\,\, -{4\pi \sqrt{-1} } \, j^*( \overline{\partial}
\partial \log w )\,\, =\,\, 4\pi \sqrt{-1} \, j^*( \partial \overline{\partial} \log w).
\end{equation}
This proves \eqref{eq:debarbetatheta}.
\end{proof}

\begin{proposition}\label{lp}
The section $\debar \beta^\theta$ (see \eqref{ej}) is a nondegenerate metric outside the hyperelliptic
locus. More precisely, the quadratic form associated to $\debar \beta^\theta$
only vanishes transversally to the hyperelliptic locus and nowhere else (see \cite[p.~223]{ACGH} for the
precise meaning of “transversally" in this context).
\end{proposition}

\begin{proof}
The Fubini--Study metric on $\mathbb{P}^{2^g-1}$ (see \eqref{eq:omegafs}) is nondegenerate everywhere,
and the map $\Theta_g$ in \eqref{tau2} is locally an immersion outside the 
locus of points associated to decomposable abelian varieties. The latter is a well-known result; see for 
example \cite[p.~620]{vgvg} (the map $\Theta_g$ is called $Th$ there). The Jacobian of a Riemann surface is 
indecomposable, and therefore $$\debar \beta^\theta(v,\overline{v})\ =\ 8\pi j^\ast \Theta_g^\ast 
\omega_{FS}(v,\overline{v})\ =\ 0$$ if and only if $\mathrm{d}j(v)\,=\,0$. This completes the proof because 
$j$ is an immersion outside the hyperelliptic locus, and it is an immersion when restricted to the 
hyperelliptic locus.
\end{proof}

\begin{remark}
\label{rem:uniformization}
From Proposition \eqref{lp} it follows immediately that the projective structure $\beta^\theta$ is different from
the projective structure given by Uniformization theorem. To see this, denote by $\beta^u$ the projective structure given
by the Uniformization theorem. It is proven in \cite{ZT} that $\debar \beta^u$ is equal to the Weil-Petersson metric, up
to a scalar multiple. Since the Weil-Petersson metric is
nondegenerate everywhere, we conclude that $\debar\beta^\theta\,
\neq\, \debar\beta^u$ on the hyperelliptic locus.
\end{remark}

We recall that $\debar \beta^\eta$ was determined up to a nonzero scalar multiple
(see \cite[Theorem 4.4]{bcfp}), and the scalar multiple was found in \cite[Theorem 5.6]{BGT}.
We present another derivation of it by an elementary computation.

\begin{proposition}\label{prop:debareta}
On the Teichm\"uller space the following holds:
\begin{equation}\label{eq:debaretabeta}
\overline{\partial} \beta^\eta\ =\ \pi j^\ast \omega_S,
\end{equation}
where
$$\omega_S\ =\ \frac{\sqrt{-1}}{2}\sum_{i,j,u,v=1}^g (\Im \tau)^{iu} (\Im \tau)^{jv} d\tau_{ij}
\wedge d\overline{\tau_{uv}}$$
is the Siegel metric on $\mathbb{H}_g$, and $j$ is the Torelli map (see \eqref{tau}).
\end{proposition}

\begin{proof}
Recall that by definition $\beta^\eta\,=\, \mathbf{r}(\eta)$. The section $\mathbf{r}(\Omega)$,
where $\mathbf r$ is the map in \eqref{eqr}, sends any
$b\, \in\, B$ to $(\Omega_b)\big\vert_{3\Delta}$, and the section
$\mathbf{r}(p_1^\ast \omega_i\wedge p_2^\ast \omega_j)\,=\,
\mathbf{m}(p_1^\ast \omega_i\wedge p_2^\ast \omega_j)$ sends any $b\, \in\, B$ to
$\omega_i(b)\omega_j(b)$ by Lemma \ref{lemma:multiplication}.

Using Lemma \ref{lemma:lambda}, identify $\omega_i(b)\omega_j(b)$ with
$-j^\ast(d\tau_{ij})(b)$. Then \eqref{eq:eta} gives the following:
\begin{equation}
\label{eq:betaeta}
\beta^\eta(b)\,=\,(\eta_b)\big\vert_{3\Delta}\,=\,\Omega_b\big\vert_{3\Delta}+\pi \sum_{i,j=1}^g (\Im j(b))^{ij}
j^\ast (d \tau_{ij})(b),
\end{equation}
and from this it is deduced that
\begin{equation}
\label{eq:2.8}
\overline\partial\beta^\eta\,\,=\,\, \pi\cdot j^\ast\left(\debar \sum_{i,j=1}^g (\Im \tau)^{ij} d \tau_{ij}\right)
\,\,=\,\,-\pi\cdot j^\ast\left(\sum_{i,j=1}^g\sum_{u\geq v}
\frac{\partial (\Im \tau)^{ij}}{\partial\, \overline{\tau_{uv}}}
d\tau_{ij}\wedge d\overline{\tau_{uv}}\right).
\end{equation}

Now,
\begin{equation}
\label{eq:debartau}
\frac{\partial (\Im \tau)^{ij}}{\partial\, \overline{\tau_{uv}}}\,=\,
- \sum_{h,l=1}^g (\Im \tau)^{ih} \frac{\partial (\Im \tau)_{hl}}{\partial\, \overline{\tau_{uv}}}(\Im \tau)^{lj}
\,=\,-\frac{\sqrt{-1}}{2}\sum_{h,l} (\Im \tau)^{ih} \frac{\partial (\overline{\tau}_{hl}-\tau_{hl})}{\partial\, \overline{\tau_{uv}}}(\Im \tau)^{lj}.
\end{equation}

Recalling that the coordinates on $\mathbb{H}_g$ are $\tau_{uv}$ with $u\,\geq\, v$, we have 
$$\frac{\partial\, \overline{\tau_{hl}}}{\partial\, \overline{\tau_{uv}}}\ =\ \begin{cases}
\delta_{hu}\delta_{lv}+\delta_{hv}\delta_{lu} & u\,>\,v\\
\delta_{hu}\delta_{lv} & u\,=\,v.
\end{cases}$$
Comparing these equations with \eqref{eq:debartau} yields
$$\frac{\partial (\Im \tau)^{ij}}{\partial\, \overline{\tau_{uv}}}\ =\ \begin{cases}
-\frac{\sqrt{-1}}{2}(\Im \tau)^{iu}(\Im \tau)^{vj}-\frac{\sqrt{-1}}{2}(\Im \tau)^{iv}(\Im \tau)^{vj} & u\,>\, v\\
-\frac{\sqrt{-1}}{2}(\Im \tau)^{iu}(\Im \tau)^{vj} & u\,=\,v.
\end{cases}$$
Plugging these into \eqref{eq:2.8} and using the fact that $\tau$ is symmetric, the proof of the proposition is completed.
\end{proof}

\section{Degenerating the metric to $\Delta_1$}\label{differ-debar}

We now prove that $\beta^\theta$ and $\beta^\eta$ (see \eqref{pr-strs}) are different for all genera. 
This will be done using degeneration to a nodal curve.

The Siegel space $\mathbb{H}_g$ has the pullback of the Fubini-Study metric of $\mathbb{P}^{2^g-1}$ ($w$ is the same as in \eqref{eq:wudoppio})
\begin{equation}\label{eq:omegatheta}
\omega_g^\Theta \ :=\ \Theta_g^\ast \,\omega^{FS}_{\mathbb{P}^{2^g-1}} \ =\ \frac{\sqrt{-1}}{2}\partial\debar \log w,
\end{equation}
where $\Theta_g$ is the map in \eqref{tau2}. Let 
\begin{equation}\label{eq:omegasiegel}
\omega_{g}^S \, :=\, \frac{\sqrt{-1}}{2}\sum_{i,j,l,h=1}^g (\Im \tau)^{il}(\Im \tau)^{jh} d \tau_{ij}\wedge d\overline{\tau_{lh}}
\end{equation}
be the Siegel metric on $\mathbb{H}_g$.

It is well known that $\omega_g^S$ descends to $\A_g$, and we proved in section \ref{sect:descend} that $\omega_g^\Theta$ descends
to $\A_g$ as well. In the following, the metrics on $\mathbb{H}_g$ that descend to $\A_g$ and the metrics they induce on $\A_g$
will be denoted --- by an abuse of notation --- with the same symbol. In Section \ref{sec:theta-proj-struct} it has 
been established that the pullbacks of $8\pi \omega_g^\Theta$ and $\pi\omega_g^S$ to $\Mg$ (through the Torelli map)
correspond to $\overline{\partial}\beta^\theta_g $ and $\overline\partial \beta^\eta_g$ respectively (see \eqref{eq:debarbetatheta}
and \eqref{eq:debaretabeta}).

The following Lemma is a stronger version of \cite[Corollary 6.12]{BGV}. Its proof uses some 
computations done in that paper.

\begin{lemma}\label{lemma:differdebar1}
On $\mathbb{H}_1$,
$$8\pi \omega_1^\Theta\ \neq\ \pi \omega_1^S,$$
or equivalently, $\debar \beta^\theta_1\,\neq\, \debar \beta^\eta_1$ on $\M_{1,1}$.
\end{lemma}

\begin{proof}
Let $z\,=\,x+y\sqrt{-1}$ be the natural coordinate on $\mathbb{H}_1$. Denote by $f_x$ the partial derivative of $f$ with respect to $x$. Note that 
$$8\pi \omega_1^\Theta \, =\, 8\pi \Theta_1^\ast \, \omega^{FS}_{\mathbb{P}^1} \, =\, 4\pi \sqrt{-1} \partial\debar \log w \, =\,
4\pi \sqrt{-1}\, \frac{w(w_{xx}+w_{yy})-w_x^2-w_y^2}{w^2}\, dz\wedge d\overline{z},$$
$$\pi \omega_1^S \ =\ \frac{\pi \sqrt{-1}}{2 y^2} \, dz\wedge d\overline{z}.$$
Comparing these two equations we obtain the following equivalence of statements:
\begin{equation}
\label{eq:inequality}
\left\{8 \pi \omega_1^\Theta(z) \,\neq\, \pi \omega_1^S(z)\right\}\, \iff\, \left\{w(z)(w_{xx}(z)+w_{yy}(z))-w_x^2(z)-w_y^2(z)
\,\neq \,\frac{w(z)^2}{2y^2}\right\}.
\end{equation}

We shall see by direct computation that the inequality in the right hand side of \eqref{eq:inequality} is valid for $x\,=\,0,\,y\,=\,1$.
This will prove the lemma.

We recall some useful formulae:
\begin{equation}
\label{eq:back1}
w(x+\sqrt{-1}y)\ =\ \sum_{m,n\in\mathbb{Z}}\exp(2\pi \sqrt{-1}xmn-\pi y(m^2+n^2)),
\end{equation}
\begin{equation}
\label{eq:back2}
\begin{cases}
w_x(\sqrt{-1})\ =\ 0,\\
w_y(\sqrt{-1})\ =\ - w(\sqrt{-1})/2.
\end{cases}
\end{equation}
These are proven in \cite[pp.~24--25]{BGV}. Equation \eqref{eq:back1} can be seen as a Corollary of the
Addition Formula \cite[p.~139]{Igusa}, while \eqref{eq:back2} corresponds to the fact that $\sigma(\sqrt{-1})\,=\,\eta(\sqrt{-1})$.

Let
$$a\,=\, \sum_{m\in\mathbb{Z}} e^{-\pi m^2},\ \ \ b\,=\, \sum_{m\in\mathbb{Z}} m^2e^{-\pi m^2},\ \ \ c\,=\,\sum_{m\in\mathbb{Z}} m^4e^{-\pi m^2}.$$
{}From \eqref{eq:back1} it follows that
\begin{gather*}
w_y(\sqrt{-1})\,=\,-2\pi ab,\qquad w_{xx}(\sqrt{-1})\,=\,-4\pi^2b^2,\\
w(\sqrt{-1})\,=\,a^2,\qquad w_{yy}(\sqrt{-1})\,=\,2\pi^2b^2+2\pi^2 ac.
\end{gather*}
With these notation \eqref{eq:back2} says that $-2\pi ab\,=\,-a^2/2$, and therefore we have $a\,=\,4\pi b$.
 
Simplifying \eqref{eq:inequality} --- using \eqref{eq:back2} and $a\,=\,4\pi b$ --- the following equivalences of statements are obtained:
\begin{gather*}
\left\{\debar \beta^\theta (\sqrt{-1})\,\neq\, \debar \beta^\eta(\sqrt{-1})\right\}\ \, \iff\\
\iff\ \left\{w(\sqrt{-1})(w_{xx}(\sqrt{-1})+w_{yy}(\sqrt{-1}))-w_y^2(\sqrt{-1})\,\neq\, \frac{w^2(\sqrt{-1})}{2}\right\}\ \iff\\
\iff\ \left\{a^2(2\pi^2ac-2\pi^2b^2)-4\pi^2a^2b^2\,\neq\,\frac{a^4}{2}\right\} \ \iff\ \left\{2\pi^2 ac\,\neq\,
\frac{a^2}{2}+6\pi^2 b^2\right\}\ \iff\\
\iff\ \left\{2\pi^2 ac\,\neq\, \frac{a^2}{2}+\frac{6}{16}a^2\right\}\ \iff\ \left\{ 16\pi^2 c\,\neq\, 7 a\right\}\
\iff\ \left\{4\pi c\,\neq\, 7b\right\};
\end{gather*}
the last inequality holds because $4\pi\,>\,7$ and $c\,>\,b$.
\end{proof}

Fix a matrix $\tau'\,\in\,\mathbb{H}_{g-1}$. The morphism
\begin{equation}
\label{eq:iota}
\iota\,:\,\mathbb{H}_1\,\lra\, \mathbb{H}_g, \qquad \tau \,\longmapsto\, \begin{pmatrix}
\tau & 0\\
0 & \tau'
\end{pmatrix}.
\end{equation}
defines an embedding of $\mathbb{H}_1$ into $\mathbb{H}_g$. We will now study the behaviour of $\omega_g^\Theta$ and $\omega_g^S$ with respect to
this map $\iota$.

\begin{lemma}\label{lemu}
Take $\tau'\,\in\,\mathbb{H}_g$ and let $\iota\,:\,\mathbb{H}_1\,\lra\, \mathbb{H}_g$ be as in \eqref{eq:iota}. Then there exists a unitary automorphism
$p$ of $\mathbb{P}^{2^g-1}$ such that the diagram
$$\begin{tikzcd}
\mathbb{H}_1 \arrow[rd, "\Theta_1"'] \arrow[r, "\iota"] & \mathbb{H}_g \arrow[r, "\Theta_g"] & \mathbb{P}^{2^g-1} \\
& \mathbb{P}^1 \arrow[ru, "p\,\circ \,e"'] &
\end{tikzcd}$$
commutes, where $e\,:\,\mathbb{P}^1\, \longrightarrow\, \mathbb{P}^{2^g-1}$ is the natural inclusion map.
\end{lemma}

\begin{proof}
For $r\,\in\, \mathbb{C}^g$ denote by $r_1\in \mathbb{C}$ and $r_2\in \mathbb{C}^{g-1}$ the vectors that satisfy
$r\,=\, \begin{pmatrix}r_1\\r_2\end{pmatrix}$. A similar notation will be used for elements of $\frac{1}{2}\mathbb{Z}_2^g
\,=\, U$. Since $\iota(\tau)$ is in block form, it follows that
\begin{equation}\label{eq:splittheta}
\theta_u(0;\iota(\tau))\ =\ \sum_{\xi \in \mathbb{Z}^g} \exp(2\pi i (\xi+u)^t \iota(\tau) (\xi+u))
\end{equation}
\[
=\ \sum_{\xi_1\in \mathbb{Z}}\exp(2\pi i (\xi_1+u_1)^t \tau (\xi_1+u_1))\sum_{\xi_2\in \mathbb{Z}^{g-1}}
\exp(2\pi i (\xi_2+u_2)^t \tau' (\xi_2+u_2))
\ =\ \theta_{u_1}(0;\tau)\theta_{u_2}(0;\tau').
\]
Consider the matrix $A\,\in\, \mathbb{C}^{2^g\times 2}$ (with rows indexed by $U$ and columns indexed by $(\frac{1}{2}\mathbb{Z})/\mathbb{Z}$)
\begin{equation}
\label{eq:matdef}
A_{u,v}\ =\ \begin{cases}
\theta_{u_2}(0;\tau') & \,\text{if }\, v\,=\,u_1\\
0 & \,\text{otherwise,}
\end{cases}
\end{equation}
and the associated linear map $A\,:\,\mathbb{C}^2\,\lra\, \mathbb{C}^{2^g}$. The map $f\,:\,\mathbb{P}^1\,\lra\, \mathbb{P}^{2^g-1}$ induced by $A$ satisfies
the condition $\Theta_g\circ \iota\,=\,f\circ \Theta_1$; indeed, this follows from \eqref{eq:splittheta} and the definition of $A$ in \eqref{eq:matdef}. 
Up to a permutation of the rows, $A$ is of the form
$$\begin{bmatrix}
a_1 & 0\\
0 & a_1\\
\vdots & \vdots\\
a_{2^g} & 0\\
0 & a_{2^g}
\end{bmatrix}$$ for some $a_1,\,\cdots,\, a_{2^g}\,\in\,\C$, not all zero. So $A$ may be completed to a matrix $U
\,\in\, \mathbb{C}^{2^g\times 2^g}$ which is a nonzero multiple of a unitary matrix. If $p\,:\,\mathbb{P}^{2^g-1}\,\longrightarrow\, \mathbb{P}^{2^g-1}$ is the automorphism
induced by $U$, then we have $f\,=\,p\circ e$. This completes the proof.
\end{proof}

\begin{corollary}\label{cor:thetaisometry}
Take $\tau'\,\in\,\mathbb{H}_g$, and let $\iota\,:\,\mathbb{H}_1\,\lra\, \mathbb{H}_g$ be as in \eqref{eq:iota}. Then
$\iota^*\om_g^\Theta\,=\, \om^\Theta_1$.
\end{corollary}

\begin{proof}
We use the notation of Lemma \ref{lemu}. If $p\,:\,\mathbb{P}^n\,\longrightarrow\, \mathbb{P}^n$ is projectively
unitary, then
$$p^\ast\omega_{\mathbb{P}^n}^{FS}\,\ =\,\ \omega_{\mathbb{P}^n}^{FS},$$ and if $e\,:\,\mathbb{P}^n\,\longrightarrow\, \mathbb{P}^{n+k}$ is the canonical inclusion
map then $e^\ast \omega_{\mathbb{P}^{n+k}}^{FS}\,=\,\omega_{\mathbb{P}^n}^{FS}$. Now it is enough to note that $$\omega_1^\Theta \ =\ \Theta_1^\ast\,\omega_{\mathbb{P}^1}^{FS} \ =\ \Theta_1^\ast\,
e^\ast p^\ast\,\omega_{\mathbb{P}^{2^g-1}}^{FS} \ =\ \iota^\ast\,\Theta_g^\ast\,\omega_{\mathbb{P}^{2^g-1}}^{FS} \ =\ \iota^\ast\,\omega_g^\Theta.$$
This completes the proof.
\end{proof}

Proving the same property for the Siegel metric is simpler. Indeed, let $\tau_{ij}$ denote the coordinates in $\mathbb{H}_g$, and
let $z$ denote the coordinate on $\mathbb{H}_1$: it is
$$\iota^\ast d\tau_{ij}\ =\ \begin{cases}
dz & \, \text{ if }\,i\,=\,j\,=\,1\\
0 & \,\text{otherwise.}
\end{cases}$$
Therefore,
\begin{equation}
\label{eq:sisometry}
\iota^\ast \omega_g^S\,=\,\iota^\ast \left(\frac{\sqrt{-1}}{2}\sum_{i,j,u,v} (\Im \tau)^{iu} (\Im \tau)^{jv} d\tau_{ij} \wedge d\overline{\tau_{uv}}\right)
\,=\,\frac{\sqrt{-1}}{2} \frac{dz \wedge d\overline{ z}}{(\Im z)^2}\,=\,\omega_1^S.
\end{equation}

With these results at hand we are ready to show that $\debar \beta^\theta\,\neq\, \debar \beta^\eta$. We turn our attention to the moduli space $\Mg$.

Let $\Mgg$ denote the Deligne--Mumford compactification of $\Mg$, which we consider as a complex analytic orbifold. Let 
$$\Delta_0,\,\,\Delta_1,\,\,\cdots, \,\, \Delta_{\lfloor \frac{g}{2}\rfloor}$$ be the divisors giving the boundary of $\Mgg$; see for example \cite[p.~50]{Ha}.
It is well-known that the period map $j$ in \eqref{tau} extends to a morphism of algebraic varieties
\begin{gather*}
\overline{j}\ :\ \Mgg - \Delta_0\ \lra\ \A_g.
\end{gather*}
Fix a smooth curve $C'$ of genus $g-1$ and fix a point $p\,\in\, C'$.
Let
\begin{gather*}
\phi\ :\ \M_{1,1} = \A_1 \,\lra\, \Delta_1 -\Delta_0 \,\subset\, \Mgg-\Delta_0, \ \ \ [E]\,\longmapsto\, [C'\cup_p E]
\end{gather*}
where $p$ is a point of $C'$, and $C' \cup_p E$ is obtained by identifying
the origin of $E$ with the point $p$. 
Also consider the map
\begin{gather*}
\iota_E\ :\ \A_1 \,\lra\, \A_g, \ \ \ [E]\, \longmapsto\, [ E \times JC']. 
\end{gather*}
We have
\begin{equation}\label{eq:iotae}
\iota_E\ =\ \overline{j} \circ \phi.
\end{equation}
Moreover, if $\tau'$ is a period matrix for $C'$ and $\iota\,:\,\mathbb{H}_1\,\lra\, \mathbb{H}_g$ is defined as
in \eqref{eq:iota}, then we have the commutative diagram
$$\begin{tikzcd}
\mathbb{H}_1 \arrow[r, "\iota"] \arrow[d] & \mathbb{H}_g \arrow[d] \\
\A_1 \arrow[r, "\iota_E"] & \A_g .
\end{tikzcd}$$

\begin{theorem}\label{teo:differ}
The following holds on $\M_g$:
\begin{gather*}
\debar \beta_g^\theta \ \neq\ \debar \beta_g^\eta .
\end{gather*}
\end{theorem}

\begin{proof}
To prove this by contradiction assume that $\debar \beta_g^\theta\, =\, \debar \beta_g^\eta$.
Then \eqref{eq:debarbetatheta} and \eqref{eq:debaretabeta} imply that (recall the notation \eqref{eq:omegatheta}, \eqref{eq:omegasiegel})
$$j^\ast \left(8\pi \omega_g^\Theta-\pi \omega_g^S\right)\ =\ 0$$
on $\M_g$, and consequently by continuity 
$$\overline{j}^* ( 8\pi \om^\Theta_g - \pi \om_{g}^S)\ =\ 0$$
on $\Mgg - \Delta_0$. It follows that $\phi^*\overline{j}^* ( 8\pi \om^\Theta_g - \pi\om_{g}^S) \,=\, 0$,
which means --- by \eqref{eq:iotae} --- that
\begin{equation}
\label{eq:ona1}
\iota_E^\ast (8\pi \omega_g^\Theta) \ =\ \iota_E^\ast (\pi \omega_g^S)
\end{equation}
on $\A_1$. Now $\iota_E^\ast\,\omega_g^S$ is induced by $\iota^\ast\,\omega_g^S$ through the quotient
map $\mathbb{H}_1\,\lra\, A_1$. In the same way $\iota_E^\ast\,\omega_g^\Theta$ is induced by $\iota^\ast\,\omega_g^\Theta$. Then \eqref{eq:ona1} implies that 
$$8\pi\, \iota^* \om^\Theta_g \ =\ \pi\,\iota^* \om_{g}^S$$
on $\mathbb{H}_1$.
But we have proved that $\iota^*\om^\Theta_g \ =\ \om^\Theta_1$ and $\iota^*\om_{g}^S \ =\ \om_{1}^S$, so it follows that
\begin{gather*}
8\pi \om^\Theta_1 \ =\ \pi \om_{1}^S
\end{gather*}
on $\mathbb{H}_1$.
This contradicts the computation of Lemma \ref {lemma:differdebar1}. This completes the proof.
\end{proof}


\begin{thebibliography}{ZZZZZ}

\bibitem[ACGH]{ACGH} E. Arbarello, M. Cornalba and P. Griffiths, {\it Geometry of algebraic 
curves}: volume II with a contribution by Joseph Daniel Harris. Vol. 268. Springer Science \& Business Media, 
2011.

\bibitem[BL]{lange-birkenhake} C.~Birkenhake and H.~Lange, \newblock
{\em Complex abelian varieties}, volume 302 of {\em Grundlehren der
Mathematischen Wissenschaften}. \newblock Springer-Verlag,
Berlin, second edition, 2004.

\bibitem[BCFP]{bcfp} I.~ Biswas, E.~Colombo, P.~Frediani and
G.P.~Pirola, Hodge theoretic projective structure on Riemann
surfaces, {\it Jour. Math. Pures. Appl.} {\bf 149} (2021), 1--27.

\bibitem[BFPT]{BFPT} I.~ Biswas, F. F.~Favale, G. P.~Pirola and
S.~Torelli, \newblock {Quillen connection and the uniformization of
Riamann surfaces}, \newblock {\it Ann. Mat. Pura
Appl.} {\bf 201} (2022), 2825--2835

\bibitem [BGT]{BGT} I.~Biswas, A.~Ghigi and C.~Tamborini, Theta
bundle, Quillen connection and the Hodge theoretic projective
structure, {\it Comm. Math. Phys.} {\bf 405} (2024), no. 10.

\bibitem [BGV]{BGV} I.~Biswas, A.~Ghigi and L.~Vai, \newblock
{Theta functions and projective structures}, {\it Int. Math. Res. Not.}
(2025), Issue 1, rnae271.

\bibitem[BR1]{BR1} I. Biswas and A. K. Raina, Projective structures on
a Riemann surface, {\it Inter. Math. Res. Not.} (1996), No. 15,
753--768.

\bibitem[BR2]{BR2} I. Biswas and A. K. Raina, Projective structures on
a Riemann surface, II, {\it Inter. Math. Res. Not.} (1999), No. 13,
685--716.

\bibitem[CFG]{EPG} E. Colombo, P. Frediani and A. Ghigi, On totally
geodesic submanifolds in the Jacobian locus, {\it Int. Jour. Math.}
{\bf 26}, no. 1 (2015) 1550005.

\bibitem[FPT]{FPT} F. F. Favale, G. P. Pirola and S. Torelli, Holomorphic
$1$-forms on the moduli space of curves, to appear in {\em Geom.
Top.} {\bf 28} (2024), 3001--3022.

\bibitem [Fay]{Fay} J. D.~Fay, \newblock {\em Theta functions on
Riemann Surfaces}, Lecture Notes in Math., Vol. 352,
Springer-Verlag, Berlin-New York, 1973.

\bibitem[GG]{vgvg} B.~van Geemen and G. van der Geer, Kummer varieties
and the moduli spaces of abelian varieties, {\it Amer. Jour. Math.}
\textbf{108} (1986), 615--642.

\bibitem[Ge]{vg} B.~van Geemen, The Schottky problem and second
order theta functions. In: {\em Workshop on Abelian Varieties and
Theta Functions} (Morelia, 1996), 41--84. Aportaciones
Mat. Investig., 13, Sociedad Matemática Mexicana, México, 1998.

\bibitem[GH]{GH} P. Griffiths and J. Harris, {\it Principles of
algebraic geometry}, Pure and Applied Mathematics,
Wiley-Interscience, New York, 1978.

\bibitem[GSM]{gradients} S. Grushevsky and R. Salvati Manni,
Gradients of odd theta functions, {\it Journal reine ang. Math.} {\bf 573} (2004), 45--59.

\bibitem[Gu1]{Gu1} R. C. Gunning, {\it On uniformization of complex
manifolds: the role of connections}, Princeton Univ. Press, 1978.

\bibitem[Gu2]{Gu2} R. C.~Gunning, \newblock {\it Some topics in the
function theory of compact Riemann surfaces},\\
\href{https://web.math.princeton.edu/~gunning/book.pdf}{\texttt{https://web.math.princeton.edu/\~{}gunning/book.pdf}}.

\bibitem[Gu3]{Gu3} R. C.~Gunning, \newblock {\it Riemann Surfaces and Second-Order Theta Functions},\\
\href{https://web.math.princeton.edu/~gunning/rsand2theta.html}{\texttt{https://web.math.princeton.edu/~gunning/rsand2theta.html}}.

\bibitem[Ha]{Ha} J. Harris and I. Morrison, {\it Moduli of curves}. Vol. 187. Springer Science \& Business 
Media, 2006.

\bibitem[Ig]{Igusa} J. I.~Igusa, \newblock {\em Theta Functions},
volume 194 of {\em Springer Science and Business Media}, {2012}.

\bibitem[Lo]{Lo} E. Looijenga, Remarkable polydifferentials on the
configuration space of a compact Riemann surface, 2021,
\href{https://webspace.science.uu.nl/~looij101/Remarkableform.pdf}{\texttt{https://webspace.science.uu.nl/~looij101/Remarkableform.pdf}}.

\bibitem[Ma]{Ma} A. Mayer, Rauch's variational formula and the heat equation,
{\it Math. Ann.} \textbf{181} (1969), 53--59.

\bibitem[Mu]{Mu} D.~Mumford, \newblock {\em Abelian varieties}, Tata
Inst. Fundam. Res. Stud. Math., 5, Published for the Tata Institute
of Fundamental Research, Bombay, by Oxford University Press, London,
1970.

\bibitem[SM]{salvati-modular} R. Salvati Manni, 
Modular varieties with level 2 theta structure, {\it Amer. Jour. Math.} {\bf 116} (1994), 1489--1511.

\bibitem [Ty]{Tyu} A. N. Tyurin, On periods of quadratic
differentials, {\em Russian Math. Surveys} {\bf 33} (1987),
169--221.

\bibitem[ZT]{ZT} P. G. Zograf and L. A. Takhtadzhyan, On the uniformization of Riemann surfaces and on 
the Weil-Petersson metric on the Teichm\"uller and Schottky spaces, {\it Math. USSR-Sb.} {\bf 60} 
(1988), 297--313.

\end{thebibliography}
\end{document}